\documentclass[12pt,leqno]{article}

\usepackage{amssymb,amsmath,amsthm}
\usepackage[all]{xy}

\usepackage{enumerate}
\usepackage{mathrsfs}
\usepackage{color}
\usepackage[colorlinks=true, pdfstartview=FitV, linkcolor=blue,%
citecolor=blue, urlcolor=blue]{hyperref}

\newcommand{\nc}{\newcommand}
\nc{\on}{\operatorname}

\newenvironment{red}{\relax\color{red}}{\hspace*{.5ex}\relax}

\newcommand{\ber}{\begin{red}}
\newcommand{\er}{\end{red}}

\newlength{\my}
\nc{\noi}{\noindent}

\newtheorem{theorem}{Theorem}[section]
\newtheorem{proposition}[theorem]{Proposition}
\newtheorem{lemma}[theorem]{Lemma}
\newtheorem{corollary}[theorem]{Corollary}
\newtheorem{convention}[theorem]{Convention}

\theoremstyle{definition}

\newtheorem{definition}[theorem]{Definition}
\newtheorem{notation}[theorem]{Notation}

\newtheorem{remark}[theorem]{Remark}

\nc{\Rem}{\begin{remark}}
\nc{\enrem}{\end{remark}}

\nc{\RR}{\mathrm{R}}
\nc{\LL}{\mathrm{L}}

\newcommand{\C}{{\mathbb{C}}}

\newcommand{\R}{{\mathbb{R}}}

\newcommand{\Z}{{\mathbb{Z}}}

\def\phi{{\varphi}}
\def\epsilon{\varepsilon}

\newcommand{\cor}{{\bf k}}

\def\sha{\mathscr{A}}
\def\shb{\mathscr{B}}
\def\shc{\mathscr{C}}
\def\shd{\mathscr{D}}
\def\she{\mathscr{E}}
\def\shf{\mathscr{F}}

\def\shm{\mathscr{M}}
\def\shn{\mathscr{N}}
\def\sho{\mathscr{O}}

\newcommand{\rmpt}{{\rm pt}}

\newcommand{\into}{\hookrightarrow}

\renewcommand{\to}[1][]{\xrightarrow[]{#1}}

\newcommand{\isoto}[1][]{\xrightarrow[#1]%
{{\raisebox{-.6ex}[0ex][-.6ex]{$\mspace{1mu}\sim\mspace{2mu}$}}}}

\newcommand{\To}{\mathop{\makebox[2em]{\rightarrowfill}}}

\newcommand{\muhom}{\mu hom}
\newcommand{\muHom}[1][]{\mathrm{Hom}^\mu_{\raise1.5ex\hbox to.1em{}#1}}
\newcommand{\Hom}[1][]{\mathrm{Hom}_{\raise1.5ex\hbox to.1em{}#1}}
\newcommand{\RHom}[1][]{\RR\mathrm{Hom}_{\raise1.5ex\hbox to.1em{}#1}}
\newcommand{\Ext}[2][]{\mathrm{Ext}_{\raise1.5ex\hbox to.1em{}#1}^{#2}}
\renewcommand{\hom}[1][]{{\mathscr{H}\mspace{-4mu}om}_{\raise1.5ex\hbox to.1em{}#1}}
\newcommand{\rhom}[1][]{{\RR\mathscr{H}\mspace{-3mu}om}_{\raise1.5ex\hbox to.1em{}#1}}
\newcommand{\rhomc}[1][]
{{\mathscr{H}\mspace{-3mu}om}^*_{\raise1.5ex\hbox to.1em{}#1}}

\newcommand{\ext}[2][]{{\mathscr{E}xt}_{\raise1.5ex\hbox to.1em{}#1}^{#2}}
\newcommand{\Tor}[2][]{\mathrm{Tor}^{\raise1.5ex\hbox to.1em{}#1}_{#2}}
\newcommand{\tens}[1][]{\mathbin{\otimes_{\raise1.5ex\hbox to-.1em{}{#1}}}}
\newcommand{\ltens}[1][]{\mathbin{\overset{\mathrm{L}}\tens}_{#1}}

\newcommand{\lltens}[1][]{{\mathop{\tens}\limits^{\rm L}}_{#1}}
\newcommand{\detens}{\underline{\etens}}

\newcommand{\etens}{\mathbin{\boxtimes}}
\newcommand{\letens}{\overset{\mathrm{L}}{\etens}}

\newcommand{\Endo}[1][]{\mathrm{End}_{\raise1.5ex\hbox to.1em{}#1}}

\newcommand{\Aut}[1][]{\mathrm{Aut}_{\raise1.5ex\hbox to.1em{}#1}}

\newcommand{\rsect}{\mathrm{R}\Gamma}

\newcommand{\conv}[1][]{\mathop{\circ}\limits_{#1}}

\newcommand{\aconv}[1][]{\mathop{\circ}\limits^{a}\limits_{#1}}

\newcommand{\sconv}[1][]{\mathop{\ast}\limits_{#1}}

\newcommand{\atimes}[1][]{\mathop{\times}\limits^{a}\limits_{#1}}

\newcommand{\oim}[1]{{#1}_*}
\newcommand{\eim}[1]{{#1}_{!}}
\newcommand{\roim}[1]{\RR{#1}_*}

\newcommand{\reim}[1]{\RR{#1}_!\mspace{2mu}}
\newcommand{\opb}[1]{#1^{-1}}

\newcommand{\epb}[1]{#1^{\,!}\,}

\newcommand{\omDA}[1][M]{\omega_{\Delta_{#1}}}
\newcommand{\omDAI}[1][M]{\omega_{\Delta_{#1}}^{{\otimes-1}}}
\newcommand{\omA}[1][M]{\omega_{#1}}
\newcommand{\omAI}[1][M]{\omega_{{#1}}^{{\otimes-1}}}
\newcommand{\dA}[1][M]{\cor_{\Delta_{#1}}}

\newcommand{\omhDA}[1][X]{\omega^{\rm hol}_{\Delta_{#1}}}

\newcommand{\omhA}[1][X]{\omega^{\rm hol}_{#1}}

\newcommand{\dhA}[1][X]{\sho_{\Delta_{#1}}}
\newcommand{\shbD}{\shb_\Delta}
\newcommand{\shcD}{\shc_\Delta}
\newcommand{\shbDU}{\shb_\Delta^\vee}
\newcommand{\shcDU}{\shc_\Delta^\vee}

\newcommand{\md[1]}{{\rm Mod}({#1})}

\newcommand{\mdcoh[1]}{{\rm Mod_{{\rm coh}}}({#1})}

\newcommand{\tr}{{\rm tr}}

\nc{\sHH}{\mathscr{H}\mspace{-4mu}\mathscr{H}}
\newcommand{\HHO}[1][X]{\sHH({\sho}_{#1})}

\newcommand{\HHD}[2][]{\sHH_{#1}(\shd_{#2})}
\newcommand{\RHHD}[2][]{\mathbb{HH}_{#1}(\shd_{#2})}
\newcommand{\RHHDk}[2][]{\mathbb{HH}^k_{#1}(\shd_{#2})}
\newcommand{\RHHDo}[2][]{\mathbb{HH}^0_{#1}(\shd_{#2})}

\newcommand{\HHE}[2][]{\sHH_{#1}(\she_{#2})}
\newcommand{\RHHE}[2][]{\mathbb{HH}_{#1}(\she_{#2})}
\newcommand{\RHHEo}[2][]{\mathbb{HH}^0_{#1}(\she_{#2})}
\newcommand{\RHHEk}[2][]{\mathbb{HH}^k_{#1}(\she_{#2})}

\newcommand{\HHCo}[2][]{\mathbb{HH}_{#1}^0({\C}_{#2})}

\nc{\sMH}{\mathscr{M}\mspace{-4mu}\mathscr{H}}
\newcommand{\MH}[2][]{\sMH_{#1}(\cor_{#2})}
\newcommand{\RMH}[2][]{\mathbb{MH}_{#1}({\cor}_{#2})}
\newcommand{\MHo}[2][]{\mathbb{MH}_{#1}^0({\cor}_{#2})}
\newcommand{\MHk}[2][]{\mathbb{MH}_{#1}^k({\cor}_{#2})}

\newcommand{\MHC}[2][]{\sMH_{#1}({\C}_{#2})}

\newcommand{\eqdot}{\mathbin{:=}}
\newcommand{\seteq}{\mathbin{:=}}

\newcommand{\cl}{\colon}
\newcommand{\scbul}{{\,\raise.4ex\hbox{$\scriptscriptstyle\bullet$}\,}}

\newcommand{\tw}[1]{\widetilde{#1}}

\newcommand{\bl}{\bigl(}
\newcommand{\br}{\bigr)}
\newcommand{\ro}{{\rm(}}
\newcommand{\rf}{\,{\rm)}}
\newcommand{\lp}{{\rm(}}
\newcommand{\rp}{{\rm)}}
\newcommand{\db}[1]{\raisebox{-.3ex}[1.5ex][1ex]{$#1$}}

\newcommand{\Rc}{{\R\text{-c}}}
\newcommand{\cc}{{\text{cc}}}

\newcommand{\GSol}{{\mathrm{Sol}}}
\newcommand{\DR}{\rm DR}

\newcommand{\ba}{\begin{array}}
\newcommand{\ea}{\end{array}}

\nc{\be}{\begin{enumerate}}
\nc{\ee}{\end{enumerate}}
\newcommand{\bnum}{\begin{enumerate}[{\rm(i)}]}
\newcommand{\enum}{\end{enumerate}}
\newcommand{\banum}{\begin{enumerate}[{\rm(a)}]}
\newcommand{\eanum}{\end{enumerate}}

\newcommand{\eq}{\begin{eqnarray}}
\newcommand{\eneq}{\end{eqnarray}}
\newcommand{\eqn}{\begin{eqnarray*}}
\newcommand{\eneqn}{\end{eqnarray*}}

\newcommand{\set}[2]{\left\{#1 \mathbin{;} #2 \right\}}

\nc{\Proof}{\begin{proof}}
\nc{\QED}{\end{proof}}
\nc{\Prop}{\begin{proposition}}
\nc{\enprop}{\end{proposition}}

\def\rop{{\rm op}}

\def\eu{{\rm eu}}
\def\hh{{\rm hh}}
\def\mueu{{\mu\rm eu}}

\newcommand{\DQ}{\ensuremath{\mathrm{DQ}}}
\newcommand{\dpi}{{\dot\pi}}

\newcommand{\HK}{\ensuremath{\mathrm{TK}}}

\DeclareMathOperator{\id}{id}

\DeclareMathOperator{\ori}{or}
\DeclareMathOperator{\chv}{char}

\newcommand{\Supp}{\on{Supp}}
\newcommand{\Der}[1][]{\mathsf{D}^{#1}}
\newcommand{\Derb}{\Der[\mathrm{b}]}

\newcommand{\SSi}{\mathrm{SS}}
\newcommand{\SSid}{\mathrm{SS}_\Delta}

\newcommand{\RD}{\mathrm{D}}
\newcommand{\RDE}{\RD_\she}
\newcommand{\RDD}{\RD_\shd}

\newcommand{\coh}{{\rm coh}}

\newcommand{\dT}{{\dot{T}}}

\nc{\eps}{\varepsilon}
\nc{\hs}{\hspace*}
\nc{\nn}{\nonumber}
\nc{\tM}{\widetilde{M}}
\nc{\h}{\mathbf{h}}
\nc{\tf}{\tilde{f}}
\nc{\codim}{\on{codim}}
\nc{\lh}{\mathscr{H}}
\nc{\bwr}{\mbox{\large{$\wr$}}}
\nc{\dTi}{\dT^{*,\mathrm{in}}}
\nc{\Cd}{\mathrm{C}}
\nc{\tK}{\widetilde{K}}
\nc{\aMM}{a_{M\times M}}

\usepackage{graphics}

\numberwithin{equation}{section}
\date{\today}

\begin{document}

\title{Microlocal Euler classes and Hochschild homology}
\author{Masaki Kashiwara%
\footnote{This work was partially supported by Grant-in-Aid for
Scientific Research (B) 22340005,
Japan Society for the Promotion of Science.}
\ and Pierre Schapira}
\maketitle

\begin{abstract}
We define the notion of a trace kernel on a manifold $M$. Roughly speaking, it is
a sheaf on $M\times M$ for which the formalism of Hochschild homology applies.
We associate a microlocal Euler class to such a kernel, a cohomology class with values in the
relative dualizing complex of the cotangent bundle $T^*M$ over $M$ and we prove that this class is
functorial with respect to the composition of kernels.

This generalizes, unifies and simplifies various results of (relative) index theorems for constructible
sheaves, $\shd$-modules and elliptic pairs.
\end{abstract}


\tableofcontents

\section{Introduction}
Our constructions mainly concern real manifolds, but in order to introduce the subject  we first consider
a complex manifold $(X,\sho_X)$. Denote by $\omhA$ the dualizing complex in the category of
$\sho_X$-modules, that is, $\omhA=\Omega_X\,[d_X]$, where
$d_X$ is the complex dimension of $X$ and $\Omega_X$ is the sheaf of
holomorphic forms of degree $d_X$.  Denote  by
$\dhA$ and $\omhDA$ the direct images of $\sho_X$ and $\omhA$ respectively by
the diagonal embedding $\delta\cl X\into X\times X$. It is well-known
(see in particular~\cite{Ca05,CaW07}) that the
Hochschild homology of $\sho_X$ may be defined by using the isomorphism
\eq\label{eq:hh1}
&&\oim{\delta}\HHO[X]\simeq\rhom[\sho_{X\times X}]\bl\dhA,\omhDA\br.
\eneq
Moreover, if $\shf$ is a coherent $\sho_X$-module and $\RD_\sho\shf\eqdot\rhom[\sho_X](\shf,\omhA)$ denotes
its dual, there are natural morphisms
\eq\label{eq:HHO}
&&\dhA\to \shf\etens\RD_\sho\shf\to\omhDA
\eneq
whose composition defines the  Hochschild class of $\shf$:
\eqn
&&\hh_\sho(\shf)\in H^0_{\Supp(\shf)}(X;\HHO[X]).
\eneqn
These constructions have been extended when replacing $\sho_X$ with a
so-called $\DQ$-algebroid stack $\sha_X$ in~\cite{KS12} ($\DQ$ stands for ``deformation-quantization'').
One of the main results  of loc.\ cit.\ is that  Hochschild classes are functorial with
respect to the composition of kernels, a kind of (relative) index
theorem for coherent $\DQ$-modules.

On the other hand, the notion of
Lagrangian cycles of constructible sheaves on real analytic manifolds
has  been introduced by the first  named author (see~\cite{Ka85})
in order to prove an index  theorem for such sheaves,
after they first appeared in the complex case (see~\cite{Ka73} and~\cite{McP74}).
We refer to~\cite[Chap.~9]{KS90} for a systematic study of Lagrangian cycles and for historical comments.
Let us briefly recall the construction.

Consider a real analytic manifold $M$ and let $\cor$ be a unital commutative ring with finite global dimension.
Denote by $\omA$ the (topological) dualizing complex of $M$, that is, $\omA=\ori_M\,[\dim M]$
where $\ori_M$
is the orientation sheaf of $M$ and $\dim M$ is the dimension.
Finally, denote by $\pi_M\cl T^*M\to M$ the cotangent bundle to $M$.
Let $\Lambda$ be a conic subanalytic Lagrangian subset of $T^*M$.
The group of Lagrangian cycles supported by $\Lambda$ is given by $H^0_\Lambda(T^*M;\opb{\pi_M}\omega_M)$.
Denote by $\Derb_\Rc(\cor_M)$ the bounded derived category of
$\R$-constructible sheaves on $M$. To an object $F$ of this category,
one associates a Lagrangian cycle supported by $\SSi(F)$, the microsupport of $F$. This cycle is
called  the characteristic cycle, or the Lagrangian
cycle or else  the {\em microlocal Euler class} of $F$ and is denoted here by $\mueu_M(F)$.

In fact, it is possible  to treat  the microlocal Euler classes of $\R$-constructi\-ble sheaves on real manifolds
similarly as the Hochschild class of coherent sheaves on complex manifolds.
Denote as above  by
$\dA$ and $\omDA$ the direct image of $\cor_M$ and $\omA$ by
the diagonal embedding $\delta_M\cl M\into M\times M$.
Then we have an isomorphism
\eq\label{eq:mueu2}
&&H^0_\Lambda(T^*M;\opb{\pi_M}\omega_M)\simeq H^0_\Lambda\bl T^*M;\muhom(\dA,\omDA)\br,
\eneq
where $\muhom$ is the microlocalization of the functor $\rhom$.
Then $\mueu_M(F)$ is obtained as follows. Denote by $\RD_MF\eqdot\rhom(F,\omA)$
the dual of $F$. There are natural morphisms
\eq\label{eq:HHEU0}
&&\dA \to F\etens\RD_MF\to\omDA,
\eneq
whose composition gives the microlocal Euler class of $F$.

In this paper, we construct the microlocal Euler class for a wide class of sheaves, including of course the
constructible sheaves but also the sheaves of  holomorphic solutions of coherent
$\shd$-modules and, more generally, of elliptic pairs  in the sense of ~\cite{ScSn94}.
To treat such situations, we are led to introduce the notion of a  trace kernel.

On a real manifold $M$ (say of class $\mathrm{C}^\infty$), a {\em trace kernel\/}
is the data of a triplet $(K,u,v)$ where $K$ is an object of the derived category of sheaves
$\Derb(\cor_{M\times M})$  and $u,v$ are morphisms
\eq\label{eq:HHEU1}
&&u\cl\dA\to K,\quad v\cl K\to\omDA.
\eneq
One then naturally defines the microlocal Euler class
$\mueu_M(K,u,v)$ of such a kernel, an element of
$H^0_\Lambda(T^*M;\muhom(\dA,\omDA))$ where
$\Lambda=\SSi(K)\cap T^*_{\Delta_M}(M\times M)$.
By~\eqref{eq:HHEU0}, a constructible sheaf gives rise to a trace kernel.

If $X$ is a complex manifold and $\shm$  is  a coherent $\shd_X$-module, we construct natural morphisms
(over the base ring $\cor=\C$)
\eq\label{eq:HHEU2}
&&\C_{\Delta_X}\to \Omega_{X\times X}\lltens[\shd_{X\times X}](\shm\detens\RD_D\shm)\to\omDA[X],
\eneq
where $\RD_D\shm$ denotes the dual of $\shm$ as a $\shd$-module.
In other words, one naturally associates a trace kernel on $X$ to a coherent $\shd_X$-module. Moreover, we
prove that under suitable microlocal conditions, the tensor product of two trace kernels is again a trace kernel,
and it follows that one can associate a trace kernel to an elliptic pair.

We  study trace kernels and their microlocal Euler classes,
showing that some  proofs of~\cite{KS12} can be easily adapted
to this situation.
One of our main results is the functoriality of the microlocal Euler classes:
the   microlocal Euler class of the composition $K_1\circ K_2$ of two trace kernels is
the composition of the microlocal Euler classes
of $K_1$ and $K_2$ (see Theorem~\ref{th:HH1} for a precise statement).
Another essential result (which is far from obvious) is that
the composition of classes coincides with
the composition for $\opb{\pi_M}\omega_M$ constructed in~\cite{KS90}
via the isomorphism between $\muhom(\dA,\omDA)$ and $\opb{\pi_M}\omega_M$.

As an application, we recover in a
single proof  the classical results on the index theorem
for constructible sheaves (see~\cite[\S~9.5]{KS90}) as well as the index theorem for elliptic
pairs of~\cite{ScSn94}, that is, sheaves of generalized holomorphic solutions of coherent
$\shd$-modules. We also briefly explain how to adapt
trace kernels to the formalism
of the Lefschetz trace formula.

We call here $\muhom(\dA,\omDA)$ the {\em microlocal homology of $M$,}
and this paper shows that, in some sense, the microlocal homology of real manifolds
plays the same role as the Hochschild homology of complex manifolds.

To conclude this introduction, let us make a general remark.
The category $\Derb_\Rc(\cor_M)$  of constructible sheaves on a compact real analytic manifold $M$ is ``proper''
in the sense of Kontsevich (that is, Ext finite) but it does not admit a Serre functor
(in the sense of Bondal-Kapranov)
and it is not clear whether it is smooth (again in the sense of Kontsevich).
However this category naturally appears in Mirror Symmetry (see~\cite{FLTZ10}) and it would be a
natural question to try to understand its  Hochschild homology in the sense of~\cite{McC94,Ke99}.
We don't know how to compute it, but the above construction,
with the use of $\muhom(\dA,\omDA)$,
provides an alternative approach of the Hochschild homology of this category. This result is not totally surprising
if one remembers the formula
(see~\cite[Prop.~8.4.14]{KS90}):
\eqn
&&\RD_{T^*M}(\muhom(F,G))\simeq\muhom(G,F)\tens \opb{\pi_M}\omega_M.
\eneqn
Hence, in some sense, $ \opb{\pi_M}\omega_M$ plays the role of a microlocal Serre functor.
Note that thanks to Nadler and Zaslow~\cite{NZ09}, the category $\Derb_\Rc(\cor_M)$ is equivalent to the
Fukaya category of the symplectic manifold $T^*M$, and this is another argument to treat
sheaves from a microlocal point of view.

\vspace{0.4ex}\noindent
{\bf Acknowledgments}\\
The second named author warmly thanks St\'ephane Guillermou for helpful discussions.

\section{A short review on sheaves}

Throughout this paper, a manifold means a  real manifold of class $C^\infty$. We
shall mainly follow the notations of~\cite{KS90} and use
some of the main notions introduced  there, in particular that of
microsupport and the functor $\muhom$.

Let $M$ be a manifold.
We denote by $\pi_M\cl T^*M\to M$ its cotangent bundle.
For a submanifold $N$ of $M$, we denote by $T^*_NM$ the
conormal bundle to $N$. In particular, $T^*_MM$ denotes the
zero-section. We set $\dT^*M\eqdot T^*M\setminus T^*_MM$ and we denote
by $\dpi_M$ the restriction of $\pi_M$ to $\dT^*M$.
If there is no risk of confusion, we write simply $\pi$ and $\dpi$ instead of $\pi_M$ and $\dpi_M$.
One denotes by $a\cl T^*M\to T^*M$ the antipodal map,
$(x;\xi)\mapsto(x;-\xi)$ and for a subset $S$ of $T^*M$, one denotes
by $S^a$ its image by this map. A set $A\subset T^*M$ is conic if it is invariant by the action
of $\R^+$ on $T^*M$.

Let $f\cl M\to N$ be a morphism of manifolds.
To $f$ one associates as usual the maps
\eq\label{eq:fdfpi}
&&\ba{c}\xymatrix@C=10ex{
T^*M\ar[dr]_-{\pi_{M}}
                 &M\times_{N}T^*N\ar[d]^\pi\ar[l]_-{f_d}\ar[r]^-{f_\pi}
                                        & T^*N\ar[d]^-{\pi_{N}}\\
                  &M\ar[r]^-f   &N.
}\ea\eneq
(Note that in loc.\ cit.\ the map $f_d$ is denoted by $\opb{{}^tf'}$.)

Let $\Lambda$ be a closed conic subset of $T^*N$. One says that $f$ is
{\em non-characteris\-tic} for $\Lambda$ if
 the map $f_d$ is proper on $\opb{f_\pi}\Lambda$ or, equivalently,
$\opb{f_\pi}\Lambda\cap\opb{f_d}(T^*_MM)\subset M\times_N T^*_NN$.

Let $\cor$ be a commutative  unital ring with finite global
homological dimension.
One denotes by $\cor_M$ the constant sheaf on $M$ with stalk $\cor$
and  by $\Derb(\cor_M)$ the bounded derived category of  sheaves of $\cor$-modules on $M$.
 When $M$ is a real analytic manifold, one denotes by
$\Derb_{\Rc}(\cor_{M})$ the full triangulated subcategory of
$\Derb(\cor_{M})$ consisting of  $\R$-constructible objects.

One denotes by $\omega_{M}$ the dualizing complex on $M$ and by $\omAI$
its dual, that is, $\omAI=\rhom(\omega_M,\cor_{M})$. More generally, for a
morphism $f\cl M\to N$, one denotes by
$\omega_{M/N}\seteq\epb{f}\cor_N\simeq\omega_{M}\tens\opb{f}(\omAI[N])$ the
relative dualizing complex.
Recall that $\omega_M\simeq\ori_M\,[\dim M]$ where  $\ori_M$ is the orientation sheaf and $\dim M$ is the
dimension of $M$. Also recall the natural morphism of functors
\eq\label{eq:opbtpepb}
&&\omega_{M/N}\tens\opb{f}\to\epb{f}.
\eneq
We have the duality functors
\eqn
&&\RD'_MF=\rhom(F,\cor_{M}),\quad \RD_MF=\rhom(F,\omega_{M}).
\eneqn

For $F\in\Derb(\cor_M)$, one denotes by $\Supp(F)$ the support of $F$ and by $\SSi(F)$ its microsupport, a
closed $\R^+$-conic co-isotropic subset of $T^*M$.
For a morphism $f\cl M\to N$ and $G\in\Derb(\cor_{N})$, one says
that $f$ is non-characteristic
for $G$ if $f$ is non-characteristic for $\SSi(G)$.

We shall use systematically the functor $\muhom$, a variant of Sato's
microlocalization functor.
Recall that for a closed submanifold $N$ of $M$, there is a  functor
$\mu_N\cl\Derb(\cor_M)\to\Derb(\cor_{T^*_NM})$ constructed by Sato (see~\cite{SKK73}) and
for  $F_1,F_2\in\Derb(\cor_M)$, one defines in~\cite{KS90} the functor
\eqn
&&\muhom\cl  \Derb(\cor_M)^\rop\times\Derb(\cor_M)\to\Derb(\cor_{T^*M}),
\\[1.5ex]
&&\hs{10ex}\muhom(F_1,F_2)\eqdot\mu_\Delta\rhom(\opb{q_2}F_1,\epb{q_1}F_2)
\eneqn
where $q_1$ and $q_2$ are the first and second projection defined on $M\times M$ and $\Delta$ is the diagonal.
This sheaf is supported by $T^*_\Delta(M\times M)$ that we identify  with
$T^*M$ by the first projection
$T^*(M\times M)\simeq T^*M\times T^*M\to T^*M$.
Note that
\eq\label{eq:suppmuh}
&&\Supp(\muhom(F_1,F_2))\subset\SSi(F_1)\cap\SSi(F_2)
\eneq
and we have Sato's\ distinguished triangle, functorial in $F_1$ and $F_2$:
\eq\label{eq:satodt}
&&\hs{1ex}\reim{\pi}\muhom(F_1,F_2)\to\roim{\pi}\muhom(F_1,F_2)
\to\roim{\dpi}\bl\muhom(F_1,F_2)\vert_{\dT^*M}\br\to[+1].
\eneq
Moreover, we have the isomorphism
\eq\label{eq:muhoma}
&&\roim{\pi}\muhom(F_1,F_2)\simeq\rhom(F_1,F_2),
\eneq
and, assuming that $M$ is real analytic and $F_1$ is $\R$-constructible, the isomorphism
\eq\label{eq:muhomb}
&&\reim{\pi}\muhom(F_1,F_2)\simeq\RD_M'F_1\ltens F_2.
\eneq
In particular, assuming that
$F_1$ is $\R$-constructible and $\SSi(F_1)\cap\SSi(F_2)\subset T^*_MM$,
we have the natural isomorphism (see~\cite[Cor~6.4.3]{KS90})
\eq\label{eq:elliptshv}
\RD'_MF_1\ltens F_2\isoto\rhom(F_1,F_2).
\eneq
As recalled in the Introduction, assuming that $M$ is real analytic, we have
the formula (see~\cite[Prop.~8.4.14]{KS90}):
\eq\label{eq:mudual}
&&\RD_{T^*M}(\muhom(F_1,F_2))\simeq\muhom(F_2,F_1)\tens \opb{\pi_M}\omega_M
\text{ for $F_1,F_2\in\Derb_\Rc(\cor_M)$.}
\eneq

\section{Compositions of kernels}
\begin{notation}\label{not:12345}
\bnum
\item For a manifold $M$,
let $\delta_M\cl M\to M\times M$ denote the diagonal embedding, and
$\Delta_M$ the diagonal set of $M\times M$.
\item Let $M_i$ ($i=1,2,3$) be manifolds. For short, we write
$M_{ij}\eqdot M_i\times M_j$ ($1\leq i,j\leq3$),
$M_{123}=M_1\times M_2\times M_3$,
$M_{1223}=M_1\times M_2 \times M_2\times M_3$, etc.

\item We will often write for short $\cor_i$ instead of $\cor_{{M_i}}$ and $\cor_{\Delta_i}$
instead of $\cor_{\Delta_{M_i}}$
and similarly with $\omega_{M_i}$, etc.,
and with the index $i$ replaced with several indices $ij$,  etc.

\item We denote by $\pi_i$, $\pi_{ij}$, etc.\ the projection
$T^*M_{i}\to M_{i}$,
$T^*M_{ij}\to M_{ij}$, etc.

\item
We denote by $q_i$ the
projection $M_{ij}\to M_i$ or the projection $M_{123}\to M_i$ and by $q_{ij}$
the projection $M_{123}\to M_{ij}$. Similarly, we denote by $p_i$ the
projection $T^*M_{ij}\to T^*M_i$ or the projection $T^*M_{123}\to T^*M_i$ and
by $p_{ij}$ the projection $T^*M_{123}\to T^*M_{ij}$.

\item We also need to
introduce the maps $p_{j^a}$ or $p_{ij^a}$, the composition of $p_{j}$ or $p_{ij}$ and the antipodal
map on $T^*M_j$.  For example,
\eqn
&&p_{12^a}((x_1,x_2,x_3;\xi_1,\xi_2,\xi_3))=(x_1,x_2;\xi_1,-\xi_2).
\eneqn
\item
We let $\delta_2\cl M_{123} \to M_{1223}$ be the natural
diagonal embedding.
\enum
\end{notation}

\medskip
We consider the operation of composition of kernels:
\eq\label{eq:conv}
&&\ba{l}
\conv[2]\;\cl\;\Derb(\cor_{M_{12}})\times\Derb(\cor_{M_{23}})\to\Derb(\cor_{M_{13}})\\
\hs{10ex}\ba{rcl}(K_1,K_2)\mapsto K_1\conv[2] K_2&\eqdot&
\reim{q_{13}}(\opb{q_{12}}K_1\ltens\opb{q_{23}}K_2)\\
&\simeq&\reim{q_{13}}\opb{\delta_2}(K_1\letens K_2).\ea
\ea
\eneq
We will use a variant of $\circ$:
\eq\label{eq:star}
&&\ba{l}
\sconv[2]\;\cl\;\Derb(\cor_{M_{12}})\times\Derb(\cor_{M_{23}})
\to\Derb(\cor_{M_{13}})\\
\hs{10ex}(K_1,K_2)\mapsto K_1\sconv[2] K_2\eqdot
\roim{q_{13}}\bl\opb{q_{2}}\omega_{2}\tens\epb{\delta_2}(K_1\letens K_2)\br.
\ea
\eneq
We also have $\omega_{M_{123}/M_{1223}} \simeq \opb{q_{2}}\omAI[M_2]$ and we
deduce from~\eqref{eq:opbtpepb} a morphism
$\opb{\delta_2} \to\opb{q_{2}}\omega_{M_2}\tens\epb{\delta_2}$\,.
Using the morphism $\reim{\,p_{13}} \to \roim{\,p_{13}}$ we obtain a natural
morphism for $K_1\in \Derb(\cor_{M_{12}})$ and $K_2\in \Derb(\cor_{M_{23}})$:
\eq
&&K_1 \conv K_2  \to K_1 \sconv K_2.
\eneq
 It is an isomorphism if
$p_{12^a}^{-1}\SSi(K_1)\cap p_{23^a}^{-1}\SSi(K_2)\to T^*M_{13}$ is proper.

We define the composition of kernels on cotangent bundles
(see~\cite[Prop.~4.4.11]{KS90})
\eq\label{eq:aconv}
&&\hs{-0ex}\ba{rcl}
\aconv[2]\;\cl\;\Derb(\cor_{T^*M_{12}})\times\Derb(\cor_{T^*M_{23}})
&\to&\Derb(\cor_{T^*M_{13}})\\
(K_1,K_2)&\mapsto&K_1\aconv[2] K_2\eqdot
\reim{p_{13}}(\opb{p_{12^a}} K_1\ltens\opb{p_{23}} K_2)\\
&&\hs{8ex}\simeq\reim{p_{13^a}}(\opb{p_{12^a}} K_1
\ltens\opb{p_{23^a}} K_2).
\ea
\eneq
We also define the corresponding operations for subsets of cotangent bundles.
Let $A\subset T^*M_{12}$ and $B\subset T^*M_{23}$. We set
\eq\ba{rcl}\label{eq:convolution_of_sets}
&&A\atimes[2]B=\opb{ p_{12^a}}(A)\cap\opb{ p_{23}}(B),\\
&&A\aconv[2] B=p_{ 13}(A\atimes[2]B)\\
&&\hs{5ex}=
\scalebox{.90}{\parbox{50ex}{
$\biggl\{\ba[c]{l}\kern-.5ex(x_1,x_3;\xi_1,\xi_3)\in T^*M_{13}\;;\;
\text{there exists
$(x_2;\xi_2)\in T^*M_2$}\\[1ex]
\hs{6ex}\text{such that $(x_1,x_2;\xi_1,-\xi_2)\in A,
(x_2,x_3;\xi_2,\xi_3)\in B$}\ea\biggr\}$.}}
\ea\eneq

We have the following result which slightly strengthens Proposition~4.4.11
of~\cite{KS90} in which the  composition $\sconv$ is not used.

\begin{proposition}\label{prop:microcompkern}
For $G_1,F_1 \in \Derb(\cor_{M_{12}})$ and $G_2,F_2 \in \Derb(\cor_{M_{23}})$
there exists a canonical morphism 
\ro whose construction is similar to that of {\rm\cite[Prop.~4.4.11]{KS90}}
\rf:
\eqn
&&\muhom(G_1,F_1)\aconv[2]\muhom(G_2,F_2)\to\muhom(G_1\sconv[2]G_2,F_1\conv[2]F_2).
\eneqn
\end{proposition}
\begin{proof}
In Proposition~4.4.8~(i) of loc.\ cit., one may replace
 $F_2\letens_S G_2$ with
$\epb{j}(F_2\letens G_2)\tens\omega_{{X\times_SY}/{X\times Y}}^{\otimes-1}$.
Then the proof goes exactly as that of Proposition~4.4.11 in loc.\ cit.
\end{proof}


\medskip
Let $\Lambda_{ij}\subset T^*M_{ij}$ \lp$i=1,2,j=i+1$\rp\, be closed conic subsets
 and consider the condition:
\eq\label{hyp:compLagr}
&&\mbox{
the projection $p_{13}\cl\Lambda_{12} \atimes[2] \Lambda_{23}\To T^*{M_{13}}$ is proper.}
\eneq
We set
\eq\label{eq:compLagr}
&&\Lambda_{13}=\Lambda_{12}\aconv[2]\Lambda_{23}.
\eneq
\begin{corollary}\label{co:microcompkern}
Assume that $\Lambda_{ij}$ \lp$i=1,2,j=i+1$\rp\,
satisfy \eqref{hyp:compLagr}. We have a composition morphism
\eqn
&&\hs{-3ex}\rsect_{\Lambda_{12}}\muhom(G_1,F_1)\aconv[2]
\rsect_{\Lambda_{23}}\muhom(G_2,F_2)\to
\rsect_{\Lambda_{13}}\muhom(G_1\sconv[2]G_2,F_1\conv[2]F_2).
\eneqn
\end{corollary}

\begin{convention}\label{conv:conv22}
In~\eqref{eq:conv}, we have introduced the composition $\conv[2]$ of kernels
$K_1\in \Derb(\cor_{M_{12}})$ and $K_2\in\Derb(\cor_{M_{23}})$. However we shall also use the notation
$M_{22}=M_2\times M_2$ and consider for example  kernels $L_1\in \Derb(\cor_{M_{122}})$ and
$L_2\in\Derb(\cor_{M_{223}})$. Then when writing $L_1\conv[2]L_2$ we mean
that the composition is taken
with respect to the last variable of $M_{22}$ for $L_1$ and the first variable for $L_2$.
In other words, set $M_4=M_2$ and consider
 $L_1$ and $L_2$ as objects of $\Derb(\cor_{M_{142}})$ and $\Derb(\cor_{M_{243}})$ respectively, in which case
the composition $L_1\conv[2]L_2$ is unambiguously defined.
\end{convention}

\section{Microlocal homology}

Let $M$ be a real manifold.
Recall that $\delta_M\cl M\hookrightarrow M\times M$ denotes the diagonal embedding.
We shall identify  $M$ with  the diagonal $\Delta_M$ of $M\times M$ and we sometimes write $\Delta$
instead of $\Delta_M$ if there is no risk of confusion. We shall identify
$T^*M$ with  $T^*_{\Delta}(M\times M)$ by the map
\eqn
&&
\delta^a_{T^*M}\cl \xymatrix@C=2.5ex{T^*M\ar@{^{(}->}[r]& T^*(M\times M)},\quad (x;\xi)\mapsto (x,x;\xi,-\xi).
\eneqn
We denote by $\dA$, $\omDA$ and $\omDAI$
the direct image
by $\delta_M$ of $\cor_M$, $\omA$ and $\omAI\seteq \rhom(\omA,\cor_M)$,
respectively.

The next definition is inspired by that of Hochschild homology on complex manifolds (see Introduction).

\begin{definition}\label{def:muHH}
Let $\Lambda$ be a closed conic subset of $T^*M$. We set
\eq&&\ba{rcl}\label{eq:mueu1B}
\MH[\Lambda]{M}&\eqdot&\rsect_\Lambda\opb{{(\delta^{a}_{T^*M})}}\muhom(\dA,\omDA),\\[.5ex]
\RMH[\Lambda]{M}&\eqdot&\rsect(T^*M;\MH[\Lambda]{M}),\\[.5ex]
\MHk[\Lambda]{M}&\eqdot& H^k(\RMH[\Lambda]{M})  =H^k(T^*M;\MH[\Lambda]{M}).
\ea\eneq
We call  $\MH[\Lambda]{M}$ the {\em microlocal homology}
of $M$ with support in $\Lambda$.
\end{definition}
We also write $\MH[]{M}$ instead of $\MH[T^*M]{M}$.

\begin{remark}\label{rem:1}
\bnum
\item
We have $\muhom(\dA,\omDA)\simeq(\delta^a_{T^*M})_*\pi_M^{-1}\omega_M$.
In particular, we have $\RMH[\Lambda]{M}\simeq\rsect_\Lambda(T^*M;\opb{\pi_M}\omega_M)$ and
$\RMH[]{M}\ \simeq\rsect(M;\omega_M)$. Assuming that $M$ is real analytic and $\Lambda$ is
a closed conic  subanalytic Lagrangian subset of
$T^*M$, we recover the space of Lagrangian cycles with support
in $\Lambda$ as  defined in~\cite[\S 9.3]{KS90}.
\item
The support of $\muhom(\dA,\omDA)$ is $T^*_{\Delta_M}(M\times M)$. Hence, we have
$\rsect_{\delta^a_{T^*M}\Lambda}\,\muhom(\dA,\omDA)\simeq(\delta^a_{T^*M})_*\MH[\Lambda]{M}$.
\item
 If $M$ is real analytic and $\Lambda$ is
a Lagrangian subanalytic closed conic subset,
then we have $H^k(\MH[\Lambda]{M})=0$ for $k<0$ (see~\cite[Prop.~9.2.2]{KS90}).
\enum
\end{remark}

 In the sequel, we denote by $\Delta_i$ (resp.\ $\Delta_{ij}$)
the diagonal subset $\Delta_{M_i}\subset M_{ii}$ (resp.\ $\Delta_{M_{ij}}\subset M_{iijj}$).

\begin{lemma}\label{le:compLagcyc}
We have natural morphisms:
\bnum
\item
$\omDA[12]\conv[22](\dA[2]\letens\omDA[3])\to\omDA[13]$,
\item
$\dA[13]\to\dA[12]\sconv[22](\omDAI[2]\letens\dA[3])$.
\enum
\end{lemma}
\begin{proof}
Denote by $\delta_{22}$ the diagonal embedding $M_{112233}\hookrightarrow M_{11222233}$.

\noindent
(i) We have the morphisms
\eqn\label{eq:compLagcyc2}\nonumber
\omDA[12]\conv[22](\dA[2]\letens\omDA[3])
    &=& \reim{q_{1133}}\opb{\delta_{22}}(\omDA[12]\letens\dA[2]\letens\omDA[3])\\
&\simeq&
\reim{q_{1133}}\omDA[123]\\
&\to&\omDA[13].
\eneqn

\noindent
(ii) The isomorphism
\eqn
&&\epb{\delta_{22}}(\cor_{\Delta_2}\etens\omDA[2])\simeq\cor_{\Delta_2}
\eneqn
gives rise to the  isomorphisms
\eqn
\dA[12]\sconv[22](\omDAI[2]\letens\dA[3])&=&
\roim{q_{1133}}\bl\opb{q_{1133}}\omA[22]\tens
\epb{\delta_{22}}(\dA[12]\letens\omDAI[2]\letens\dA[3])\br\\
&\simeq&
\roim{q_{1133}}\epb{\delta_{22}}(\dA[1]\letens\omDA[2]\letens\dA[23]) \\
&\simeq&\roim{q_{1133}}\dA[123]
\eneqn
and the result follows by adjunction from the morphism
\eqn
&&\opb{q_{1133}}\dA[13]\simeq \dA[1]\letens\cor_{22}\letens \dA[3]
\to\dA[1]\letens\dA[2]\letens \dA[3]=\dA[123].
\eneqn
\end{proof}

\begin{proposition}\label{prop:complagcyc}
Let $M_i$ \lp$i=1,2,3$\rp\, be manifolds. 
We have a natural composition morphism 
\ro whose constructions will be given in the course  
of the proof{\rm\/):}
\eq\label{eq:compLagcyc}
&&\muhom(\dA[12],\omDA[12])\aconv[22]\muhom(\dA[23],\omDA[23])
\to
\muhom(\dA[13],\omDA[13]).
\eneq
In particular, let $\Lambda_{ij}$ be a closed conic subset of $T^*M_{ij}$ \lp $ij=12,13,23$\rp.
If $\Lambda_{12}\aconv[2]\Lambda_{23}\subset  \Lambda_{13}$, then we have a morphism
\eq\label{eq:compLagcyc-sectionsa}
\MH[\Lambda_{12}]{12}&\aconv[2]&\MH[\Lambda_{23}]{23}\to \MH[\Lambda_{13}]{13}.
\eneq
\end{proposition}
\begin{proof}
Consider the morphism (see Proposition~\ref{prop:microcompkern} and Convention~\ref{conv:conv22})
\eqn
\muhom(\omDAI[2],\omDAI[2])\aconv[2]\muhom(\dA[23],\omDA[23])
   &\to&\muhom(\omDAI[2]\sconv[2]\dA[23],\omDAI[2]\conv[2]\omDA[23])\\
&\simeq&\muhom(\omDAI[2]\letens\dA[3],\dA[2]\letens \omDA[3]).
\eneqn
It induces an isomorphism
\eq
&&\muhom(\dA[23],\omDA[23])
\simeq \muhom(\omDAI[2]\letens\dA[3],\dA[2]\letens \omDA[3]).\label{eq:mhomtwist}
\eneq
 Note that this isomorphism is also obtained by
\eqn
\muhom(\dA[23],\omDA[23])&\simeq&
\muhom\bl(\omAI[2]\letens \cor_{233})\ltens \dA[23],
(\omAI[2]\letens \cor_{233})\ltens \omDA[23]\br\\
&\simeq&
\muhom(\omDAI[2]\letens\dA[3],\dA[2]\letens \omDA[3]).
\eneqn
Applying Proposition~\ref{prop:microcompkern}, we get a morphism:
\eq\label{eq:compLagcyc1}\nonumber
\muhom(\dA[12],\omDA[12])\aconv[22]\muhom(\dA[23],\omDA[23])&&\\
&&\hspace{-20ex}\to
\muhom(\dA[12]\sconv[22](\omDAI[2]\letens\dA[3]),
\omDA[12]\conv[22](\dA[2]\letens \omDA[3])).
\eneq
It remains to apply Lemma~\ref{le:compLagcyc}.
\end{proof}

\begin{corollary}\label{co:compLC}
Let   $\Lambda_{ij}$ \lp$i=1,2,j=i+1$\rp\,
satisfying~\eqref{hyp:compLagr} and let
$\Lambda_{13}=\Lambda_{12}\aconv[2]\Lambda_{23}$.
The composition of kernels  in \eqref{eq:compLagcyc-sectionsa}
induces a morphism
\eq\label{eq:compLagcyc-sectionsb}
\aconv[2]\hs{.5ex}\;:\;\RMH[\Lambda_{12}]{12}&\ltens&\RMH[\Lambda_{23}]{23}\to\RMH[\Lambda_{13}]{13}.
\eneq
In particular, each $\lambda\in\MHo[\Lambda_{12}]{12}$ defines a morphism
\eq\label{eq:compLagcyc-sectionsb2}
\lambda\aconv[2]\hs{.5ex}\cl\RMH[\Lambda_{23}]{23}\to\RMH[\Lambda_{13}]{13}.
\eneq
\end{corollary}
\begin{proof}
These morphisms follow from~\eqref{eq:compLagcyc-sectionsa}.
The second assertion follows from the isomorphism
$H^0(X)\simeq\Hom[{\Derb(\cor)}](\cor,X)$ in the category $\Derb(\cor)$.
\end{proof}

\begin{theorem}\label{th:muconv}
\bnum
\item We have the isomorphisms
\eqn
\muhom(\dA,\omDA)&\simeq&
(\delta_{T^*M}^a)_*\pi_M^{-1}\rhom(\cor_M,\omega_M)\\
&\simeq&(\delta_{T^*M}^a)_*\pi_M^{-1}\omega_M.
\eneqn
\item
We have a commutative diagram
\eq
&&\ba{l}\xymatrix{
\muhom(\dA[12],\omDA[12])\aconv[22]\muhom(\dA[23],\omDA[23])\ar[r]\ar[d]^{\bwr}&
\muhom(\dA[13],\omDA[13])\ar[d]^{\bwr}\\
(\delta_{T^*M_{13}}^a)_*\bl\pi_{M_{12}}^{-1}\omega_{M_{12}}\aconv[2]
\pi_{M_{23}}^{-1}\omega_{M_{23}}\br\ar[r]&(\delta_{T^*M_{13}}^a)_*
\pi_{M_{13}}^{-1}\omega_{M_{13}}.
}\ea\label{dia:muconv}
\eneq
Here 
the top horizontal arrow of \eqref{dia:muconv}
is given  in {\rm Proposition~\ref{prop:complagcyc}}, and
the bottom horizontal arrow is induced by
\eqn
&&\hspace{-5ex}p_{12^a}^{-1}\pi_{M_{12}}^{-1}\omega_{M_{12}}\ltens \opb{p_{23}}\pi_{M_{23^a}}^{-1}\omega_{M_{23}}
\simeq \pi_{M_1}^{-1}\omega_{M_1}\letens
\pi_{M_2}^{-1}(\omega_{M_2}\ltens\omega_{M_2})\letens
\pi_{M_3}^{-1}\omega_{M_3},\\
&&\hspace{-5ex}\pi_{M_2}^{-1}(\omega_{M_2}\ltens\omega_{M_2})\simeq\omega_{T^*M_2},\\
&&\hspace{-5ex}\reim{p_{13}}\bl \pi_{M_1}^{-1}\omega_{M_1}\letens
\omega_{T^*M_2}\letens\pi_{M_3}^{-1}\omega_{M_3}\br
\To \pi_{M_1}^{-1}\omega_{M_1}\letens \pi_{M_3}^{-1}\omega_{M_3}.
\eneqn
\enum
\end{theorem}
\Proof
(i) is obvious.

\medskip\noindent 
(ii)--(a) By~\cite[Prop.~4.4.8]{KS90}, we have natural morphisms for $(i,j)=(1,2)$ or $(i,j)=(2,3)$:
\eqn
\muhom(\cor_{\Delta_{{i}}},\omDA[{i}])\letens\muhom(\cor_{\Delta_{{j}}},\omDA[{j}])&\to&\muhom(\cor_{\Delta_{{ij}}},\omDA[{ij}])
\eneqn
and it follows from (i) that these morphisms are isomorphisms. These isomorphisms give rise to the isomorphism
\eqn
\muhom(\dA[12],\omDA[12])\aconv[22]\muhom(\dA[23],\omDA[23])&\simeq&\\
&&\hs{-40ex}
\muhom(\dA[1],\omDA[1])\letens\bl\muhom(\dA[2],\omDA[2])\aconv[22]\muhom(\dA[2],\omDA[2])\br\letens\muhom(\dA[3],\omDA[3])
\eneqn
Similarly, we have an isomorphism
\eqn
\pi_{M_{12}}^{-1}\omega_{M_{12}}\aconv[2]\pi_{M_{23}}^{-1}\omega_{M_{23}}&\simeq&
\pi_{M_{1}}^{-1}\omega_{M_{1}}\etens  \bl \pi_{M_{2}}^{-1}\omega_{M_{2}}\aconv[2]\pi_{M_{2}}^{-1}\omega_{M_{2}}\br\etens\pi_{M_{1}}^{-1}\omega_{M_{1}}.
\eneqn
Hence, we are reduced to the case where $M_1=M_3=\rmpt$, 
which we shall assume now.

\medskip\noindent  
(ii)--(b) We change our notations and we set:
\eqn
&&M\eqdot M_2, \quad Y\eqdot M\times M,\\
&&\delta_M\cl M\into Y\mbox{ the diagonal embedding},\,\Delta_M=\delta_M(M),\\
&&j\cl Y\into Y\times Y \mbox{ the diagonal embedding, $\Delta_Y=\delta_Y(Y)$,}\\
&&\delta_{T^*M}^a\cl T^*M\into T^*Y, \, (x;\xi)\mapsto (x,x;\xi,-\xi),\\
&& \delta_{T^*Y}^a\cl T^*Y\into T^*Y\times T^*Y,\\
&&p\cl T^*Y\to\rmpt\mbox{ the projection,}\\
&&a_Y\cl Y\to\rmpt\mbox{ the projection.}
\eneqn
With these new notations, the composition $\aconv[22]$ will be denoted  by  $\aconv[T^*Y]$.

Consider the diagram~\ref{diag:4415} similar to Diagram~(4.4.15) of ~\cite{KS90}
\eq\label{diag:4415} 
&&\ba{l}\xymatrix{
T^*M\times T^*M\ar@{^{(}->}[r]^-i
&T^*Y\times T^*Y&\,T^*Y\ar@^{_{(}->}[l]_-{\delta_{T^*Y}^a}\ar[rrdd]^(.6)p&\\
&T^*Y\times_Y T^*Y\ar[u]_-{j_\pi}\ar[d]^-{j_d}
&\,T^*_{\Delta_Y}(Y\times Y)\ar@^{_{(}->}[l]_-{\tw s}\ar[d]^-{\pi_Y}\ar[u]_-{p_1}^-\sim\\
&T^*Y\ar@{}[ur]|-{\displaystyle\Box}&Y\ar[l]_s\ar[rr]_-{a_Y}&&\rmpt.
}\ea\eneq
Here, $i$ is the canonical embedding
induced by $\delta_{T^*M}^a$, 
$p_1$ is induced by the first projection $T^*Y\times T^*Y\to T^*Y$, $s\cl Y\into T^*Y$ is the zero-section embedding and $\tw s$ is the natural embedding. Note that the square labelled by $\square$ is Cartesian. 
We have 
\eqn
\reim{p}\circ{(\delta_{T^*Y}^a)}^{-1}&\simeq& \reim{a_Y}\circ \reim{\pi_Y}\circ\opb{p_1}\circ\opb{{(\delta_{T^*Y}^a)}}\\
&\simeq&  \reim{a_Y}\circ \reim{\pi_Y}\circ\opb{\tw s}\circ\opb{j_\pi}\\
&\simeq& \reim{a_Y}\circ \opb{s}\circ \reim{j_d}\circ \opb{j_\pi}.
\eneqn
Therefore,
\eqn
&&\muhom(\dA[M],\omDA[M])\aconv[T^*Y]\muhom(\dA[M],\omDA[M])\\
&&\hs{8ex}\simeq
\reim{p}{(\delta_{T^*Y}^a)}^{-1}
\bl\muhom(\dA[M],\omDA[M])\letens\muhom(\dA[M],\omDA[M])\br\\
&&\hs{16ex}\simeq\reim{a_Y}\opb{s}\reim{j_d}\opb{j_\pi}\muhom(\dA[M]\letens\dA[M],\omDA[M]\letens\omDA[M]).
\eneqn
Hence, by adjunction, to give a morphism
\eqn
&&\muhom(\dA[M],\omDA[M])\aconv[T^*Y]\muhom(\dA[M],\omDA[M])\to\cor
\eneqn
is equivalent to giving a morphism in $\Derb(\cor_{Y})$
\eq
&&
 \opb{s}\reim{j_d}\opb{j_\pi}\muhom(\dA[M]\letens\dA[M],\omDA[M]\letens\omDA[M])
\to\epb{a_Y}\cor_\rmpt.
\label{eq:12}
\eneq
Note that the left hand side of \eqref{eq:12}
is supported on $\Delta_M$.
Hence in order to give a morphism \eqref{eq:12},
it is necessary and sufficient to give a morphism
in $\Derb(\cor_{M})$
\eq
&&
\opb{{\delta_{M}}}
\opb{s}\reim{j_d}\opb{j_\pi}\muhom(\dA[M]\letens\dA[M],\omDA[M]\letens\omDA[M])
\to \epb{{\delta_{M}}}\epb{a_Y}\cor_\rmpt.
\label{eq:13}
\eneq
Hence, it is enough to check the commutativity of 
the upper square in the following diagram in $\Derb(\cor_M)$ 
\eq
&&\hs{-5ex}\ba{l}\xymatrix{
\opb{{\delta_{M}}}\opb{s}\reim{j_d}\opb{j_\pi}\muhom(\dA[M]\letens\dA[M],\omDA[M]\letens\omDA[M])\ar[r]\ar[d]^-\sim
&\epb{\delta_{M}}\epb{a_Y}\cor_\rmpt\ar[d]^-\id\\
\opb{{\delta_{M}}}\opb{s}\reim{j_d}\opb{j_\pi}i_*
\bl\opb{\pi_M}\omega_M\letens\opb{\pi_M}\omega_M\br\ar[r]\ar[d]^-\sim
&\epb{\delta_{M}}\epb{a_Y}\cor_\rmpt\ar[d]^-\sim\\
\omega_M\ar[r]^{\id}&\omega_M.
}\ea\label{dia:muconvB}
\eneq
The top horizontal arrow is constructed by  the chain of morphisms (see~\cite[\S~4.4]{KS90}):
\eqn
&&\ba{l}\reim{j_d}\opb{j_\pi}\muhom(\dA[M]\letens\dA[M],\omDA[M]\letens\omDA[M])\\
\hs{15ex}\to\muhom(\epb{j}(\dA[M]\letens\dA[M])\ltens\omega_Y,\opb{j}(\omDA[M]\letens\omDA[M]))\\
\hs{15ex}\simeq\muhom(\omDA[M],\omDA[M]\tens \omDA[M])\simeq 
{(\delta_{T^*M}^a)}_*\opb{\pi_M}\omega_M
\ea\eneqn
and 
\eq\label{eq:morks90}
&&\ba{l}
\opb{{\delta_{M}}}\opb{s}\reim{j_d}\opb{j_\pi}
\muhom(\dA[M]\letens\dA[M],\omDA[M]\letens\omDA[M])\\
\hs{15ex}\to \opb{{\delta_{M}}}\opb{s}{(\delta_{T^*M}^a)}_*\opb{\pi_M}\omega_M
\simeq\omega_M.\ea\eneq
Hence, the commutativity of the diagram~\eqref{dia:muconvB} is reduced to
the commutativity of the diagram below:
\eq
&&\hs{-5ex}\ba{l}\xymatrix{
\opb{\delta_{M}}\opb{s}\reim{j_d}\opb{j_\pi}\muhom(\dA[M]\letens\dA[M],\omDA[M]\letens\omDA[M])\ar[d]\ar[drr]^-\lambda\\
\opb{{\delta_{M}}}\opb{s}\reim{j_d}\opb{j_\pi}i_*
\bl\opb{\pi_M}\omega_M\letens\opb{\pi_M}\omega_M\br
\ar[rr]^-\sim&&\omega_M.
}\ea\label{dia:muconvC}
\eneq
where the morphism $\lambda$ is given by the morphisms in~\eqref{eq:morks90}.
All terms of \eqref{dia:muconvC}  are concentrated at the degree $-\dim M$.
Hence the commutativity of \eqref{dia:muconvC}
is a local problem in $M$ and we can assume that $M$ is a Euclidean space.
We can checked directly in this case.
\QED

\begin{remark}\label{rem:muconv2}
Theorem~\ref{th:muconv} may be applied as follows.
Let $\Lambda_{ij}$ be a closed conic subset of $T^*M_{ij}$ ($i=1.2$, $j=i+1$).
Assume~\eqref{hyp:compLagr}, that is, the projection $p_{13}\cl\Lambda_{12} \atimes[2] \Lambda_{23}\To T^*{M_{13}}$ is proper and set
$\Lambda_{13}=\Lambda_{12}\aconv[2]\Lambda_{23}$.
Let $\lambda_{ij}\in\MHo[\Lambda_{ij}]{M_{ij}}\simeq H^0_{\Lambda_{ij}}(T^*M_{ij};\opb{\pi}\omega_{ij})$. Then
\eq\label{eq:convcap}
&&\lambda_{12}\aconv[2]\lambda_{23}=\int_{T^*M_2}\lambda_{12}\cup\lambda_{23}
\eneq
where the right hand-side is obtained as follows.
Set $\Lambda\eqdot\Lambda_{12}\atimes[2]\Lambda_{23}$ and consider the morphisms
\eqn
&&H^0_{\Lambda_{12}}(T^*M_{12};\opb{\pi}\omega_{12})\times H^0_{\Lambda_{23}}(T^*M_{23};\opb{\pi}\omega_{23})\\
&&\hspace{10ex}\to H^0_\Lambda(T^*M_{123};\opb{\pi}\omega_{1}\letens\omega_{T^*M_2}\letens\opb{\pi}\omega_{3})\\
&&\hspace{10ex}\to H^0_{\Lambda_{13}}(T^*M_{13};\opb{\pi}\omega_{13}).
\eneqn
The first morphism is the cup  product and the second one is the integration morphism with respect to $T^*M_2$.
\end{remark}

\section{Microlocal Euler classes of trace kernels}
In this section, we often write $\Delta$ instead of $\Delta_M$.
\begin{definition}\label{def:traceker}
A {\em trace kernel} $(K,u,v)$ on $M$ is the data of $K\in\Derb(\cor_{M\times M})$ together with morphisms
\eq\label{def:dp1}
&&\dA[]\to[\ u\ ] K\quad \text{and}\quad K\to[\ v\ ]\omDA[].
\label{eq:dpmor01}\eneq
\end{definition}
In the sequel, as far as there is no risk of confusion, we simply write $K$ instead of $(K,u,v)$.

For a trace kernel $K$ as above, we set
\eq\label{def:SSidp}
&&\SSid(K)\eqdot\SSi(K)\cap T^*_\Delta (M\times M)=(\delta_{T^*M}^a)^{-1}\SSi(K).
\eneq
(Recall that one often identifies $T^*M$ and  $T^*_\Delta(M\times M)$ by
$\delta_{T^*M}^a\cl T^*M\hookrightarrow T^*M\times T^*M$.)

\begin{definition}\label{def:dpmueu}
Let $(K,u,v)$ be a trace kernel.
\banum
\item
The morphism $u$ defines  an element
$\tilde{u}$ in $H^0_{\SSid(K)}(T^*M;\muhom(\dA[],K))$
and the {\em microlocal Euler class} $\mueu_M(K)$ of $K$ is the image
of  $\tilde{u}$ by the morphism
$\muhom(\dA[],K)\to \muhom(\dA[],\omDA[])$
associated with the morphism $v$.
\item
 Let $\Lambda$ be a closed conic subset of $T^*M$ containing $\SSid(K)$. One denotes
by $\mueu_\Lambda(K)$ the image of
 $\tilde{u}$ in $H^0_\Lambda\bl T^*M;\muhom(\dA[],\omDA[])\br$.
\eanum
\end{definition}
Hence,
\eq\label{eq;defmueu}
&&  \mueu_\Lambda(K)\in\MHo[\Lambda]{M}\simeq H^0_{\Lambda}(T^*M;\opb{\pi}\omega_M).
\eneq

Let $\tilde{v}$ be the element of $H^0_{\SSid(K)}(T^*M;\muhom(K,\omDA[]))$
induced by $v$.
Then the microlocal Euler class $\mueu_M(K)$ of $K$ coincides with the image
of $\tilde{v}$ by the morphism
$\muhom(K,\omDA)\to \muhom(\dA[],\omDA[])$
associated with the morphism $u$, which can be easily seen by the commutative diagram:

\eqn
&&\xymatrix{
(\delta_{T^*M}^a)^{-1}\muhom(K,K)\ar[r]^-v\ar[d]^-u
&(\delta_{T^*M}^a)^{-1}\muhom(K,\omDA[])\ar[d]^-u\\
(\delta_{T^*M}^a)^{-1}\muhom(\dA[],K)\ar[r]^-v
&(\delta_{T^*M}^a)^{-1}\muhom(\dA[],\omDA[]).
}
\eneqn
One denotes by $\eu(K)$ the restriction of $\mueu(K)$ to the zero-section $M$ of $T^*M$ and
calls it the {\em Euler class} of $K$. Hence
\eq\label{eq;defeu}
&&  \eu_M(K)\in H^0_{\Supp(K)\cap\Delta}(M;\omega_M).
\eneq
It is nothing but the class induced by the composition
$\dA[M]\to K\to \omDA[M]$.

We say that $L\in\Derb(\cor_M)$ is {\em invertible} if
$L$ is locally isomorphic to $\cor_M[d]$ for some $d\in\Z$.
Then, $L^{\otimes-1}\seteq\rhom(L,\cor_M)$
is also invertible and $L\ltens L^{\otimes-1}\simeq\cor_M$.

\Prop\label{prop:ambg}
Let $L$ be an invertible object in $\Derb(\cor_M)$
and $K$ a trace kernel.
Then $K\ltens(L\letens L^{\otimes -1})$ is a trace kernel and
$\mueu\bl K\ltens(L\letens L^{\otimes -1})\br=\mueu(K)$.
\enprop
\Proof
$L\letens L^{\otimes -1}$ is canonically isomorphic to $\cor_{M\times M}$ on a neighborhood of
the diagonal set $\Delta_M$ of $M\times M$.
\QED

\begin{remark}
Of course, we could also have defined a trace kernel
as a sequence of morphisms
\eq\label{eq:traceomDAI}
&&\omDAI\to \tK\to\dA.
\eneq
When treating sheaves, both  definitions   would give the same microlocal Euler class
by taking $K=\tK\tens(\cor_M\letens\omega_M)$. However, when working with $\sho$-modules or with
$\DQ$-modules as in~\cite{KS12}, the two constructions give different classes. Note that
we have chosen an analogue of \eqref{eq:traceomDAI}  in \cite{KS12}.
\end{remark}

\subsubsection*{Trace kernels for constructible sheaves}
Let us denote by  $\Derb_\cc(\cor_M)$ the full  triangulated  subcategory of
$\Derb(\cor_M)$ consisting of cohomologically constructible  sheaves (see~\cite[\S~3.4]{KS90}).

\begin{lemma}\label{lem:HHmor12}
Let $F\in\Derb_\cc(\cor_M)$.
There are natural morphisms in  $\Derb_\cc(\cor_{M\times M})$:
\eq
&&\dA\to F\letens\RD_M F,\label{eq:HHmor01}\\
&&F\letens\RD_M F\to\omDA.\label{eq:HHmor02}
\eneq
\end{lemma}
In other words, an object $F\in\Derb_\cc(\cor_M)$ defines naturally a trace kernel on $M$.
\begin{proof}
(i) We have
\eqn
\cor_M&\to&\rhom(F,F)\simeq \epb{\delta}(F\letens \RD_M F).
\eneqn
Hence, the result follows by adjunction.

\noindent
(ii) The morphism \eqref{eq:HHmor02} may be deduced from~\eqref{eq:HHmor01} by duality, or
by adjunction  from the morphism
\eqn
\opb{\delta}(F\letens \RD_M F)\to\omega_M.
\eneqn
\end{proof}
\begin{notation}
We shall denote by $\HK(F)$ the trace kernel associated with $F\in\Derb_\cc(\cor_M)$, that is the data of
$F\letens \RD_M F$ and the morphisms~\eqref{eq:HHmor01},~\eqref{eq:HHmor02}.
Note that we have always $\SSid(\HK(F))\subset\SSi(F)$ and
the equality holds
if $M$ is real analytic and $F$ is $\R$-constructible.
\end{notation}

We have the chain of morphisms
\eqn
\muhom(F,F)
&\simeq& \opb{{(\delta^{a}_{T^*M})}}\muhom(\cor_\Delta,F\letens\RD F)\\
&\to& \opb{{(\delta^{a}_{T^*M})}}\muhom(\cor_\Delta,\omDA[]).
\eneqn
We deduce the map
\eq\label{eq:defeuler}
&&H^0_{\SSi(F)}(T^*M;\muhom(F,F))\to \MHo[\SSi(F)]{M}.
\eneq

\begin{definition}\label{def:lagcyc}
Let $F\in\Derb_\cc(\cor_M)$.
The image of $\id_F$ by the map \eqref{eq:defeuler} is called the microlocal Euler class of $F$ and
is denoted  by $\mueu_M(F)$.
\end{definition}
Clearly, one has
\eq\label{eq:euF=euHKF}
&&\mueu_M(F)=\mueu_M(\HK(F)).
\eneq
Assume $M$ is real analytic and denote by $\Derb_\Rc(\cor_M)$ the  full triangulated subcategory of
$\Derb(\cor_M)$ consisting  of $\R$-constructible complexes.
Of course, $\R$-constructible  complexes  are
cohomologically constructible.
In~\cite[\S~9.4]{KS90} the microlocal Euler class of an object $F\in\Derb_\Rc(\cor_M)$ is constructed
as above and this class
is also called the characteristic cycle, or else, the Lagrangian cycle, of $F$.

\begin{remark}
Let $(K,u,v)$ be a trace kernel on $M$.
Let $\delta\cl M\to M\times M$ be the diagonal embedding.
Then $u$ and $v$ decompose as
\eqn
&&\dA\to \oim{\delta}\epb{\delta}K\to K\to \oim{\delta}\opb{\delta}K\to \omDA.
\eneqn
Hence $\delta_*\epb{\delta}K$ and $\oim{\delta}\opb{\delta}K$
are also trace kernels. We have evidently
\eqn
&&\mueu_M\bl\oim{\delta}\epb{\delta}K\br=\mueu_M\bl \oim{\delta}\opb{\delta}K\br=\mueu_M(K)
\quad\text{ as elements in $\MHo[T^*M]{M}$.}
\eneqn
\end{remark}

\subsubsection*{Trace kernels over one point}
Let us consider the particular case where $M$ is a single point, $M=\rmpt$, and let us identify a sheaf
over $\rmpt$ with a $\cor$-module.
In this situation, a trace kernel $(K,u,v)$  is the data of $K\in\Derb(\cor)$ together with linear maps
\eqn
&&\cor\to[u]K\to[v]\cor.
\eneqn
The (microlocal) Euler class $\eu_\rmpt(K)$ of this kernel is the image of $1\in\cor$ by $v\circ u$.

Assume now that $\cor$ is a field and denote by $\Derb_f(\cor)$
the full triangulated subcategory of $\Derb(\cor)$ consisting of objects
with finite-dimensional cohomologies.
Let $V\in\Derb_f(\cor)$ and set $V^*=\RHom(V,\cor)$. Let
$K=\HK(V)=V\tens V^*$, and let $v$ be the trace morphism and $u$ its dual.  Then
\eq
&&\parbox{60ex}{
\banum
\item
$\eu_\rmpt(V\tens V^*)=\tr(\id_V)$, the trace of the identity of $V$.
\item If $\cor$ has characteristic zero, then\\[1ex]
\hs{4ex}$\eu_\rmpt(V\tens V^*)=\chi(V)$, the Euler-Poincar\'e index of $V$.
\ee}\label{eq:euler index}
\eneq

\subsubsection*{Trace kernels for $\shd$-modules}
In this subsection, we denote by $X$ a complex manifold of complex dimension $d_X$ and the
base ring $\cor$ is the field $\C$. We denote by $\sho_X$ the structure sheaf and by
 $\Omega_X$ the sheaf of holomorphic forms of maximal degree.
We still denote by $\omega_X$ the topological dualizing complex
 and recall the isomorphism $\omega_X\simeq\C_X\,[2d_X]$.

One denotes by $\shd_X$ the sheaf of $\C_X$-algebras of (finite order) holomorphic differential operators on $X$
 and  refers  to~\cite{Ka03} for a detailed exposition of the theory of $\shd$-modules. We denote by
$\md[\shd_X]$ the category of left $\shd_X$-modules and by $\Derb(\shd_X)$ its bounded derived category.
We also denote by $\mdcoh[\shd_X]$ the abelian category of coherent $\shd_X$-modules and by $\Derb_\coh(\shd_X)$
the full triangulated subcategory of $\Derb(\shd_X)$ consisting of objects with coherent  cohomologies.

We denote by $\RDD\cl \Derb(\shd_X)^\rop\to\Derb(\shd_X)$
the duality functor for left $\shd$-modules:
\eqn
&&\RDD\shm\eqdot\rhom[\shd_X](\shm,\shd_X)\tens[\sho_X]\Omega_X^{\tens-1}\,[d_X].
\eneqn
We denote by $\scbul\detens\scbul$ the external product for $\shd$-modules:
\eqn
&&\shm\detens\shn\eqdot \shd_{X\times X}\tens[\shd_X\letens\shd_X](\shm\letens\shn).
\eneqn
Let $\Delta$ be the diagonal of $X\times X$.
The left $\shd_{X\times X}$-module $H^{d_X}_{[\Delta]}(\sho_{X\times X})$ (the algebraic cohomology
with support in $\Delta$) is denoted as usual by $\shbD$. Note that
\eqn
&&\RDD \shbD\simeq \shbD.
\eneqn
One shall be aware that here, the dual is taken over $X\times X$.
We also introduce
\eqn
&&\shbDU\eqdot\shbD\,[2d_X].
\eneqn
 For $\shm\in \Derb_\coh(\shd_X)$, we have the isomorphism
\eqn
\rhom[\shd_X](\shm,\shm)  &\simeq&\rhom[\shd_{X\times X}](\shbD,\shm\detens\RDD\shm)\,[d_X].
\eneqn
We deduce the morphism  in $\Derb(\shd_{X\times X})$
\eq\label{eq:DHHmor01}
&&\shbD\to\shm\detens\RDD\shm\,[d_X]
\eneq
and by duality, the morphism in $\Derb(\shd_{X\times X})$
\eq\label{eq:DHHmor02}
&&\shm\detens\RDD\shm\,[d_X]\to\shbDU.
\eneq
Denote by $\she_{X}$ the sheaf on $T^*X$ of microdifferential operators of ~\cite{SKK73}.
For a coherent $\shd_X$-module $\shm$ set
\eqn
&&\shm^E\eqdot \she_{X}\tens[\opb{\pi}\shd_X]\opb{\pi}\shm
\eneqn
and recall that, denoting by  $\chv(\shm)$ the characteristic variety of $\shm$, we have
 $\chv(\shm)=\Supp(\shm^E)$. One also sets
\eqn
&&\shc_\Delta\eqdot\shb_\Delta^E,\quad\shcDU\seteq\bl\shbDU\br^E.
\eneqn
We denote by $\RDE\cl\Derb(\she_X)^\rop\to\Derb(\she_X)$ the duality functor for left $\she$-modules:
\eqn
&&\RDE\shm\eqdot\rhom[\she_X](\shm,\she_X)\tens[\opb{\pi}\sho_X]\opb{\pi}\Omega_X^{\tens-1}\,[d_X]
\eneqn
and we denote by $\scbul\detens\scbul$ the external product for $\she$-modules:
\eqn
&&\shm\detens\shn\eqdot \she_{X\times X}\tens[\she_X\letens\she_X](\shm\letens\shn).
\eneqn
The morphisms ~\eqref{eq:DHHmor01} and~\eqref{eq:DHHmor02} give rise to the morphisms
\eq  \label{eq:EHHmor0}
&&\shcD\to \shm^E\detens\RDE\shm^E\,[d_X]\to\shcDU.
\eneq
Let $\Lambda$ be a closed conic subset of $T^*X$.
One sets
\eqn
&&\HHE[]{X}=\opb{{(\delta^{a}_{T^*X})}}\rhom[\she_{X\times X}](\shcD,\shcDU),\\
&&\RHHE[\Lambda]{X}=\rsect_\Lambda(T^*X;\HHE[]{X}),\\
&&\RHHEk[\Lambda]{X}=H^k(\RHHE[\Lambda]{X})=H^k_\Lambda(T^*X;\HHE[]{X}).
\eneqn
We call $\HHE[\Lambda]{X}$ the {\em Hochschild homology}
 of $\she_X$ with support in $\Lambda$.

The morphisms in~\eqref{eq:EHHmor0} define a class
\eq\label{eq:hhM}
&&\hh_\she(\shm)\in\RHHEo[\chv(\shm)]{X}
\eneq
that we call the {\em Hochschild class} of $\shm$.

Let $S$ be a closed subset of $X$.
By restricting to the zero-section $X$ of $T^*X$ the above construction, we obtain the  Hochschild homology
of $\shd_X$:
\eqn
\HHD[]{X}&=&\opb{{(\delta_{X})}}\rhom[\shd_{X\times X}](\shbD,\shbDU)
\simeq\HHE[]{X}\vert_X,\\
\RHHD[S]{X}&=&\rsect_S(X;\HHD[]{X}),\\
\RHHDk[S]{X}&=&H^k(\RHHD[S]{X})=H^k_S(X;\HHD[]{X}).
\eneqn
Then, to $\shm\in\Derb_\coh(\shd_X)$ one  obtains 
\eqn
&& \hh_\shd(\shm)\eqdot\hh_\she(\shm)\vert_X\in\RHHDo[\Supp(\shm)]{X}.
\eneqn

We shall make a link between the Hochschild class of $\shm$ and the microlocal Euler class of a
trace kernel attached to the sheaves of holomorphic solutions of $\shm$. We need a lemma.
\begin{lemma}\label{le:homEmuhom}
For $\shn_1$ and $\shn_2$ in $\Derb_\coh(\shd_X)$, there exists a natural morphism
\eq\label{eq:homEmuhom}
&&\rhom[\she](\shn_1^E,\shn_2^E)\to\muhom(\Omega_X\ltens[\shd_X]\shn_1,\Omega_X\ltens[\shd_X]\shn_2).
\eneq
Moreover, this morphism is compatible with the composition
\eqn
&&\rhom[\she](\shn^E_1,\shn^E_2)\tens\rhom[\she](\shn_2^E,\shn_3^E)\to\rhom[\she](\shn_1^E,\shn_3^E),\\
&&\muhom(F_1,F_2)\tens\muhom(F_2,F_3)\to\muhom(F_1,F_3).
\eneqn
\end{lemma}
\begin{proof}
We have  the natural morphism
in $\Derb(\opb{\pi}\shd_X\tens\opb{\pi}\shd_X^\rop)$
 (see~\cite[Prop.~10.6.2]{KS85})
 \eqn
 &&\she_X\to\muhom(\Omega_X,\Omega_X).
 \eneqn
 This gives rise to the morphisms
 \eqn
 \rhom[\opb{\pi}\shd_X](\opb{\pi}\shn_1,\she_X\tens[\opb{\pi}\shd_X]\opb{\pi}\shn_2)&&\\
 &&\hspace{-30ex}
 \to \rhom[\opb{\pi}\shd_X](\opb{\pi}\shn_1,\muhom(\Omega_X,\Omega_X))\tens[\opb{\pi}\shd_X]\opb{\pi}\shn_2\\
 &&\hspace{-30ex}
 \simeq \muhom(\Omega_X\ltens[\shd_X]\shn_1,\Omega_X\ltens[\shd_X]\shn_2).
 \eneqn
 \end{proof}
We have
\eqn
\Omega_{X\times X}\,[-d_X]\ltens[\shd_{X\times X}]\shbD&\simeq&\C_\Delta,\\
\Omega_{X\times X}\,[-d_X]\ltens[\shd_{X\times X}] \shbDU&\simeq&\omDA[].
\eneqn
Applying Lemma~\ref{le:homEmuhom}, one deduces the morphisms
\eqn
\rhom[\she_{X\times X}](\shcD,\shcDU)
&\to&
\muhom(\Omega_{X\times X}\ltens[\shd_{X\times X}]\shbD,\Omega_{X\times X}
\ltens[\shd_{X\times X}]\shbDU)\\
 &\simeq&\muhom(\C_\Delta,\omDA[]).
\eneqn
An easy calculation shows that the first arrow  is also an isomorphism. Therefore, we get the isomorphism
\eq\label{eq:HHEHH}
&&\HHE[]{X}\isoto\MHC[]{X}.
\eneq
Recall that the Hochschild homology of $\she_X$ has been already calculated in~\cite{BG87}.

Applying the functor $\Omega_{X\times X}\,[-d_X]\ltens[\shd_{X\times X}]\scbul$ to~\eqref{eq:DHHmor01}
and~\eqref{eq:DHHmor02} we get the morphisms
\eq\label{eq:HHEU3}
&&\C_\Delta\to\Omega_{X\times X}\lltens[\shd_{X\times X}](\shm\detens\RDD\shm)\to\omDA[].
\eneq
\begin{notation}\label{not:HKM}
For $\shm\in\Derb_\coh(\shd_X)$, we denote by $\HK(\shm)$ the trace kernel given by~\eqref{eq:HHEU3}.
\end{notation}
Since $\chv(\shm)=\SSi(\rhom[\shd_X](\shm,\sho_X))$ by~\cite[Th.~11.3.3]{KS90}, we get that
$\mueu_M(\HK(\shm))$ is supported by $\chv(\shm)$, the characteristic variety of $\shm$.

\begin{proposition}
After identifying $\HHE[]{X}$ and $\MHC[]{X}$ by the isomorphism~\eqref{eq:HHEHH},
we have $\hh_\she(\shm)=\mueu_X(\HK(\shm))$ in $\HHCo[\chv(\shm)]{X}$.
\end{proposition}
\begin{proof}
This follows from Lemma~\ref{le:homEmuhom} applied to~\eqref{eq:EHHmor0}.
\end{proof}
Note that the class  $\mueu_X(\HK(\shm))$ coincides
with the microlocal Euler class of $\shm$ already introduced by Schapira-Schneiders in~\cite{ScSn94}.

\section{Operations on microlocal Euler classes I}
In this section, we shall adapt to  trace kernels the constructions of \cite[Chap.~4~\S3]{KS12} and we shall
show that under natural microlocal conditions of properness, the microlocal Euler class of the
composition of two kernels is the  composition of the classes.

We use Notations~\ref{not:12345} and we consider  a trace kernel $(K,u,v)$ on $M_{12}$.
\begin{lemma}\label{le:HH3}
Let $K$ be a trace kernel on $M_{12}$.
There are natural morphisms in $\Derb(\cor_{M_{11}})$:
\eq
&&\dA[13]\to K\sconv[22](\omDAI[2]\letens\dA[3]),\label{eq:HHmor1}\\
&& K\conv[22](\dA[2]\letens\omDA[3])\to\omDA[13].\label{eq:HHmor2}
\eneq
\end{lemma}
\begin{proof}
(i) By  Lemma~\ref{le:compLagcyc}~(ii) we have a morphism
$\dA[13]\to \dA[12]\sconv[22](\omDAI[2]\letens\dA[3])$. By composing this morphism with $\dA[12]\to K$,
we get~\eqref{eq:HHmor1}.

\vspace{2ex}\noindent
(ii)  By  Lemma~\ref{le:compLagcyc}~(i) we have a morphism $\omDA[12]\conv[22](\dA[2]\letens\omDA[3])\to\omDA[13]$.
 By composing this morphism with $K\to\omDA[12]$ we we get~\eqref{eq:HHmor2}.
\end{proof}

Let $K$ be  a trace kernel on $M_{12}$ with microsupport $\SSi(K)$ contained in a closed conic subset
$\Lambda_{1122}$ of $T^*M_{1122}$ and let $\Lambda_{23}$ a closed conic subset of $T^*M_{23}$.
We assume
\eq\label{hyp:lambdapproper}
&&
\Lambda_{1122}\atimes[22]\delta_{T^*M_{23}}^a\Lambda_{23} \mbox{ is proper over }T^*M_{1133}.
\eneq
We set
\eq\label{eq:lambda1}
&&\left\{
\parbox{60ex}{
$\Lambda_{12}\eqdot\Lambda_{1122}\cap T^*_{\Delta_{12}}M_{1122}$,\\
$\Lambda_{1133}\eqdot\Lambda_{1122}\aconv[22]\delta_{T^*M_{23}}^a\Lambda_{23}$,\\
$\Lambda_{13}\eqdot\Lambda_{1133}\cap T^*_{\Delta_{13}}M_{1133}= \Lambda_{12}\aconv[2]\Lambda_{23}$.
}\right.
\eneq

We define a map
\eq\label{eq:defPhiK}
&&\Phi_K
\cl\RMH[\Lambda_{23}]{23}\To\RMH[\Lambda_{13}]{13}
\eneq
by the sequence of morphisms
\eqn
&&\RMH[\Lambda_{23}]{23}
\simeq \rsect_{\delta_{T^*M_{23}}^a\Lambda_{23}}(T^*M_{2233}\,;\,\muhom(\dA[23],\omDA[23]))\\
&&\hspace{3ex}\simeq \rsect_{\delta_{T^*M_{23}}^a\Lambda_{23}}\bl T^*M_{2233}\,;\,
\muhom(\omDAI[2]\letens\dA[3],\dA[2]\letens\omDA[3])\br\\
&&\hspace{3ex}\to \rsect_{\Lambda_{1133}}\bl T^*M_{1133}\,;\,
\muhom(K,K)\aconv[22]\muhom(\omDAI[2]\letens\dA[3],\dA[2]\letens\omDA[3])\br\\
&&\hspace{3ex}\to\rsect_{\Lambda_{1133}}\bl T^*M_{1133}\,;\,
\muhom(K\sconv[22](\omDAI[2]\letens\dA[3]),K\conv[22](\dA[2]\letens\omDA[3]) )  \br\\
&&\hspace{3ex}\to\Gamma\bl T^*M_{1133}\,;\,\muhom(\dA[13],\omDA[13])\br
\simeq\RMH[{\Lambda_{13}}]{13}.
\eneqn
 Here the first arrow is given by $\id_{K}$, the second is given by
 Proposition~\ref{prop:microcompkern},
and\ the last arrow is  induced by the morphisms in Lemma~\ref{le:HH3}.

The next result is similar to \cite[Th.~4.3.5]{KS12}.
\begin{proposition}\label{pro:HH4}
Let  $\Lambda_{1122}\subset T^*M_{1122}$ and $\Lambda_{23}\subset T^*M_{23}$ be closed conic subsets
satisfying~\eqref{hyp:lambdapproper} and recall the notation~\eqref{eq:lambda1}.
Let $K$ be a trace kernel on $M_{12}$ with microsupport contained in $\Lambda_{1122}$.
Then the map $\Phi_K$  in \eqref{eq:defPhiK} is the map $\mueu_{M_{12}}(K)\aconv[12]\;$ given by Corollary~\ref{co:compLC}.
\end{proposition}
\begin{proof}

By using the morphism $\dA[12]\to K$, we find the commutative diagram below:

\eqn
&&\hspace*{-8ex}\kern -20ex
\scalebox{.78}{\parbox{\my}{
$\xymatrix@C=2ex{
\rsect_{\Lambda_{23}}
\bl T^*M_{2233};\muhom(\dA[23],\omDA[23])\br \ar[r]\ar[d]&
\rsect_{\Lambda_{13}}
\bl T^*M_{1133};\muhom(\dA[12]\sconv[22]\dA[23],\dA[12]\conv[22]\omDA[23])\br\ar[d]\\
\rsect_{\Lambda_{1133}}\bl T^*M_{1133};
\muhom(K\sconv[22]\dA[23],K\conv[22]\omDA[23])\br\ar[r]&
\rsect_{\Lambda_{13}}
\bl T^*M_{1133};\muhom(\dA[12]\sconv[22]\dA[23], K\conv[22]\omDA[23])\br.
}$}}\eneqn
By using the morphism $K\to\omDA[12]$, we get the commutative diagram
\eq
&&\hspace{-10ex}\kern-10ex
\scalebox{.75}{\parbox{\my}{$
\xymatrix@C=-20ex@R=5ex{
\rsect_{\Lambda_{23}}\bl T^*M_{2233};\muhom(\dA[23],\omDA[23])\br \ar[rr]\ar[rd]&&
 \db{\rsect_{\Lambda_{13}}\bl T^*M_{1133};
\muhom(\dA[12]\sconv[22]\dA[23],\omDA[12]\conv[22]\omDA[23])\br}.\\
&\rsect_{\Lambda_{1133}}\bl T^*M_{1133};
\muhom( K\sconv[22]\dA[23], K\conv[22]\omDA[23])\br\ar[ru]&
}
$}
}
\label{diag:hhconv}
\eneq

\noindent
Recall the morphisms in Lemma~\ref{le:compLagcyc}:
\eq\ba{rcl}\label{eq:CtoFGtoC}
&&\omDA[12]\conv[22](\dA[2]\letens\omDA[3])\to\omDA[13],\quad \dA[13]\to\dA[12]\sconv[22](\omDAI[2]\letens\dA[3]).
\ea\eneq
We get the morphisms
\eqn
&&\hspace{-4.3ex}w\cl\rsect_{\delta_{T^*M_{13}}^a\Lambda_{13}}\bl T^*M_{1133};
\muhom(\dA[12]\sconv[22]\dA[23],\omDA[12]\conv[22]\omDA[23])\\
&&\hspace{-1ex}\simeq\rsect_{\delta_{T^*M_{13}}^a\Lambda_{13}}\bl T^*M_{1133};
\muhom(\dA[12]\sconv[22](\omDAI[2]\letens\dA[3]),\omDA[12]\conv[22](\dA[2]\letens\omDA[3]))\br\\
&&\hspace{27ex}\to\rsect_{\delta_{T^*M_{13}}^a\Lambda_{13}}\bl T^*M_{1133};\muhom(\dA[13],\omDA[13])\br.
\eneqn
By its construction, the morphism $\mueu_{M_{12}}(K)\conv\;$ is obtained as the
composition with the map $w$ of the top row of the diagram~\eqref{diag:hhconv}.
Since the composition with $w$ of the two other arrows is the morphism  $\Phi_K$, the proof is complete.
\end{proof}
The next result is similar to \cite[Th.~4.3.6]{KS12}.

Let $i=1,2$, $j= i+1$ and let $\Lambda_{iijj}$ be a closed conic subset of $T^*M_{iijj}$.
Assume that
\eq\label{hyp:112233proper}
&&\Lambda_{1122}\atimes[22]\Lambda_{2233} \mbox{ is proper over }T^*M_{1133}.
\eneq
Set $\Lambda_{1133}=\Lambda_{1122}\aconv[22]\Lambda_{2233}$ and $\Lambda_{ij}=\Lambda_{iijj}\cap T^*_{\Delta_{ij}}M_{iijj}$.

\begin{theorem}\label{th:HH1}
Let $K_{ij}$ be a trace kernel on $M_{ij}$ with $\SSi(K_{ij})\subset\Lambda_{iijj}$.
Assume~\eqref{hyp:112233proper}, set $\tK_{23}=\omDAI[2]\conv[2]K_{23}
\simeq(\omAI[2]\letens\cor_{233})\ltens K$  and set
$K_{13}=K_{12}\conv[22]\tK_{23}$. Then
\banum
\item
$K_{13}$ is a trace kernel on $M_{13}$,
\item
 $\mueu_{M_{13}}(K_{13})=\mueu_{M_{12}}(K_{12})\aconv[2] \mueu_{M_{23}}(K_{23})$
as elements of $\MHo[\Lambda_{13}]{13}$.
\item
In particular, we have $\Phi_{K_{12}}\circ\Phi_{K_{23}}\simeq\Phi_{K_{13}}$.
\eanum
\end{theorem}

\begin{proof}
(a)  The trace kernel $K_{23}$ defines morphisms
\eqn
&&\omDAI[2]\letens\dA[3]\to \tK_{23}\to\dA[2]\letens\omDA[3].
\eneqn
Assuming~\eqref{hyp:112233proper} and using~\eqref{eq:HHmor1} and~\eqref{eq:HHmor2}, we get that
$K_{13}=K_{12}\conv[22]\tK_{23}$ is a trace kernel on $M_{13}$.

\vspace{2ex}\noindent
(b) We get a commutative diagram in which we set
$\lambda_{23}=\mueu_{M_{23}}(K_{23})\in\MHo[]{23}\simeq\Hom(\omDAI[2]\letens\dA[3],\dA[2]\letens\omDA[3])$:
\eqn
\xymatrix{
\dA[13]\ar[r]\ar `d[ddr] [ddr]
   &{\db{K_{12}\sconv[22](\omDAI[2]\letens\dA[3])}}
\ar[r]^-{\lambda_{23}}\ar[d]
&{\db{ K_{12}\conv[22](\dA[2]\letens\omDA[3])}}\ar[r]&\omDA[13]\\
&K_{12}\sconv[22]\tK_{23}&&\\
&K_{12}\conv[22]\tK_{23}  \ar[ruu]
\ar`r[rruu][rruu]\ar[u]^-\bwr
}\eneqn
The composition of the arrows on the bottom is
$\mueu_{M_{13}}(K_{13})$ and the  composition of the
arrows on the top is $\Phi_{K_{12}}(\mueu_{M_{23}}(K_{23}))$.
Hence, the assertion follows from
the commutativity of the diagram by Proposition~\ref{pro:HH4}.

\noindent
(c) follows from (b) and Proposition~\ref{pro:HH4}.
\end{proof}

\section{Operations on microlocal Euler classes II}

We shall combine Theorems~\ref{th:muconv} and~\ref{th:HH1} and make more explicit
the operations on microlocal Euler classes for direct or inverse images.
In particular, applying our results to the  case of constructible sheaves,
we shall recover the results of~\cite[Ch.~IX~\S5]{KS90}.

Let $M$ be a manifold and let $\iota\cl N\into M$ be closed embedding of a smooth submanifold $N$.
If there is no risk of confusion, we shall still denote by $\cor_N$ and $\omega_N$ the sheaves
$\oim{\iota}\cor_N$ and $\oim{\iota}\omega_N$ on $M$.
Then $\cor_N$ is cohomologically  constructible and moreover
\eqn
&&\RD_M\cor_N= \rhom(\cor_N,\omega_M)\simeq \omega_{N}.
\eneqn
Hence,  $\HK(\cor_N)=\cor_N\letens\omega_{N}$ is a trace kernel on $M$.

 Let $M_i$ be a manifold ($i=1,2$), let $K_i$ be a trace kernel on $M_i$ and let
$\Lambda_{ii}$ be a closed conic subset of $T^*M_{ii}$ with $\SSi(K_i)\subset\Lambda_{ii}$.
We set
\eqn
&&\Lambda_i=\Lambda_{ii}\cap T^*_{\Delta_i}M_{ii}.
\eneqn
For  a morphism of manifolds $f\cl M_1\to M_2$, we  denote by $\Gamma_f$ its graph, a smooth closed
submanifold of $M_{12}$ and we set for short
\eqn
&&\Lambda_f\eqdot T^*_{\Gamma_f}(M_{12}),\quad \tw f=(f,f)\cl M_{11}\to M_{22}.
\eneqn
Recall the diagram~\eqref{eq:fdfpi}
\eqn\label{eq:fdfpi2}
&&\xymatrix@C=10ex{
T^*M_1\ar[dr]_-{\pi_{M_1}}
                 &M_1\times_{M_2}T^*M_2\ar[d]^\pi\ar[l]_-{f_d}\ar[r]^-{f_\pi}
                                        & T^*M_2\ar[d]^-{\pi_{M_2}}\\
                  &M_1\ar[r]^-f   &M_2.
}\eneqn
Note that
\eqn
\Lambda_{11}\aconv[11]\Lambda_{\tw f}=\tw f_\pi\opb{\tw f_d}\Lambda_{11},\quad
\Lambda_{\tw f}\aconv[22]\Lambda_{22}=\tw f_d\opb{\tw f_\pi}\Lambda_{22}.
\eneqn
In the sequel, we shall identify $M_{1212}$ with $M_{1122}$.
We take as kernel the sheaf $\HK(\cor_{\Gamma_f})$. Then
\eq\label{eq:HKgammaf}
\HK(\cor_{\Gamma_f})&=&\cor_{\Gamma_f}\letens\omega_{\Gamma_f}\simeq
\cor_{\Gamma_{\tw f}}\tens(\cor_1\letens\omega_1\letens\cor_{22})\\
&\simeq&\omDA[1]\conv[11]
\bl(\omega_1^{\otimes-1}\letens\omega_1\letens \cor_{22})
\ltens\cor_{\Gamma_{\tw f}}\br.\nonumber
\eneq
Moreover, we have (see~\eqref{eq:euF=euHKF}):
\eqn
&&\mueu_{M_{12}}(\HK(\cor_{\Gamma_f}))=\mueu_{M_{12}}(\cor_{\Gamma_f}).
\eneqn
Also note that
\eqn
\reim{\tw f}K_1\simeq K_1\conv[11]\cor_{\Gamma_{\tw f}},\quad
\opb{\tw f}K_2\simeq \cor_{\Gamma_{\tw f}}\conv[22]K_2.
\eneqn

\subsubsection*{External product}
Applying Theorem~\ref{th:muconv} with $M_2=\rmpt$ and $M_3$ being here $M_2$,
we get the commutative diagram
\eqn
\xymatrix{
\MH[\Lambda_1]{M_1}\letens\MH[\Lambda_2]{M_2}\ar[r]^-{\conv}\ar[d]^-\sim&\MH[\Lambda_1\times\Lambda_2]{M_{12}}\ar[d]^-\sim\\
\rsect_{\Lambda_1}(\opb{\pi_{M_1}}\omega_{M_1})\letens\rsect_{\Lambda_2}(\opb{\pi_{M_2}}\omega_{M_{2}})\ar[r]^-\letens
                       &\rsect_{\Lambda_1\times\Lambda_2}(\opb{\pi_{M_{12}}}\omega_{M_{12}})
}\eneqn
and taking the global sections and the $0$-th cohomology,
\eqn
\xymatrix{
\MHo[\Lambda_1]{M_1}\tens\MHo[\Lambda_2]{M_2}\ar[r]^-\conv\ar[d]^-\sim&\MHo[\Lambda_1\times\Lambda_2]{M_{12}}\ar[d]^-\sim\\
H^0_{\Lambda_1}(T^*M_1;\opb{\pi_{M_1}}\omega_{M_1})\tens H^0_{\Lambda_2}(T^*M_2;\opb{\pi_{M_2}}\omega_{M_{2}})\ar[r]^-\letens
                       &H^0_{\Lambda_1\times\Lambda_2}(T^*M_{12};\opb{\pi_{M_{12}}}\omega_{M_{12}}).
}\eneqn
Applying Theorem~\ref{th:HH1}, we obtain
\begin{proposition}\label{pro:extmueu}
The object $K_1\letens K_2$ is a trace kernel on $M_{12}$ and
\eqn
&&\mueu_{M_{12}}(K_1\letens K_2)=\mueu_{M_1}(K_1)\letens\mueu_{M_2}(K_2).
\eneqn
\end{proposition}

\subsubsection*{Direct image}
Let $f\cl M_1\to M_2$  and $\Gamma_f$ be as above.
Applying Theorem~\ref{th:muconv}  with $M_1=\rmpt$ and $M_2,\, M_3$ 
being the current $M_1,\, M_2$, 
we get the commutative diagram
\eqn
\xymatrix{
\MH[]{M_1}\aconv[1]\MH[]{M_{12}}\ar[r]\ar[d]^-\sim&\MH[]{M_{2}}\ar[d]^-\sim\\
\opb{\pi_{M_1}}\omega_{M_1}\aconv[1]\opb{\pi_{M_{12}}}\omega_{M_{12}}\ar[r]
                       &\opb{\pi_{M_2}}\omega_{M_2}.
}\eneqn
Now we assume
\eq\label{hyp:fproper}
&&\parbox{65ex}{
$f$ is proper on  $\Lambda_{1}\cap T^*_{M_{1}}M_{1}$, or, equivalently, $f_\pi$ is proper on $\opb{f_d}\Lambda_1$.
}
\eneq
We set
\eqn
&&f_\mu(\Lambda_1)=\Lambda_1\circ\Lambda_f= f_\pi(\opb{f_d}(\Lambda_1)).
\eneqn
Taking the global sections and the $0$-th cohomology of the diagram above, we obtain the commutative diagram:
\eqn
\xymatrix{
\MHo[\Lambda_1]{M_1}\ar[rr]^{\circ \mueu(\cor_{\Gamma_f})}\ar[d]^-\sim&&\MHo[f_\mu\Lambda_1]{M_{2}}\ar[d]^-\sim\\
H^0_{\Lambda_1}(T^*M_1;\opb{\pi_{M_1}}\omega_{M_1})\ar[rr]^-{\circ \mueu(\cor_{\Gamma_f})}
                       &&H^0_{f_\mu\Lambda_1}(T^*M_2;\opb{\pi_{M_{2}}}\omega_{M_{2}}).
}\eneqn
We have natural morphism and isomorphisms, already constructed in \cite{KS90}:
\eqn
\eim{f_\pi}\opb{f_d}\opb{\pi_{M_1}}\omega_{M_1}&\simeq&\eim{f_\pi}\opb{\pi}\omega_{M_1}
\simeq\opb{\pi_{M_2}}\eim{f}\omega_{M_1}\\
&\to&\opb{\pi_{M_2}}\omega_{M_2}.
\eneqn
It induces a morphism:
\eqn
&&f_\mu\cl \rsect_{\Lambda_1}(\opb{\pi_{M_1}}\omega_{M_1})\to\rsect_{f_\mu\Lambda_1}(\opb{\pi_{M_2}}\omega_{M_2}).
\eneqn
\begin{lemma}
Let $\lambda\in H^0_{\Lambda_1}(T^*M_1;\opb{\pi_{M_1}}\omega_{M_1})$. Then
$\lambda\circ \mueu_{M_{12}}(\cor_{\Gamma_f})=f_\mu(\lambda)$.
\end{lemma}

\begin{proposition}\label{pro:directmueu}
Assume that $\tw f$ is proper on $\Lambda_{11}\cap T^*_{M_{11}}M_{11}$.
Then the object $\reim{\tw f} K_1$ is a trace kernel on $M_2$ and
\eqn
\mueu_{M_2}(\reim{\tw f}K_1)
                &=&\mueu_{M_1}(K_1)\aconv[1]\mueu_{M_{12}}(\cor_{\Gamma_f})\\
                &=&f_\mu(\mueu_{M_1}(K_1)).
\eneqn
\end{proposition}
\begin{proof}
Note that
$\mueu_{M_{12}}(\cor_{\Gamma_f})=
\mueu_{M_{12}}\bl(\omega_1^{\otimes-1}\letens\omega_1\letens\cor_{22})\ltens
\HK(\cor_{\Gamma_f})\br$ by Proposition~\ref{prop:ambg}.
We have $\reim{\tw f} K_1\simeq K_1\conv[11]
\Bigl(\omDAI[1]\conv[1]\bl(\omega_1^{\otimes-1}\letens\omega_1\letens\cor_{22})\ltens
\HK(\cor_{\Gamma_f})\br\Bigr)$. It remains to
 apply Theorem~\ref{th:HH1} in which one replaces $M_1, M_2, M_3$ with $\rmpt, M_1,M_2$, respectively.
\end{proof}

\subsubsection*{Inverse image}
Let $f\cl M_1\to M_2$  and $\Gamma_f$ be as above.
Applying Theorem~\ref{th:muconv}
 with $M_3=\rmpt$,  we get the commutative diagram

\eqn
\xymatrix{
\MH[]{M_{12}}\aconv[2]\MH[]{M_{2}}\ar[r]\ar[d]^-\sim&\MH[]{M_{1}}\ar[d]^-\sim\\
\opb{\pi_{M_{12}}}\omega_{M_{12}}\aconv[2]\opb{\pi_{M_{2}}}\omega_{M_{2}}\ar[r]
                       &\opb{\pi_{M_1}}\omega_{M_1}.
}\eneqn
Now we assume

\eq\label{hyp:fnonchar}
&&\parbox{65ex}{
$f$ is non characteristic for $\Lambda_2$, or, equivalently, $f_d$ is proper on $\opb{f_\pi}\Lambda_2$.
}\eneq
We set
\eqn
&&f^\mu(\Lambda_2)=\Lambda_f\circ\Lambda_1= f_d(\opb{f_\pi}(\Lambda_2)).
\eneqn
Taking the global sections and the $0$-th cohomology of the diagram above, we obtain the commutative diagram:
\eqn
\xymatrix{
\MHo[\Lambda_2]{M_2}\ar[rr]^{\mueu(\cor_{\Gamma_f})\circ}\ar[d]^-\sim&&\MHo[f^\mu\Lambda_2]{M_{1}}\ar[d]^-\sim\\
H^0_{\Lambda_2}(T^*M_2;\opb{\pi_{M_2}}\omega_{M_2})\ar[rr]^-{\mueu(\cor_{\Gamma_f})\circ}
                       &&H^0_{f^\mu\Lambda_2}(T^*M_1;\opb{\pi_{M_{1}}}\omega_{M_{1}}).
}\eneqn
We have a natural morphism  constructed in the proof of~\cite[Prop.~9.3.2]{KS90}:
\eqn
f^\mu\cl\eim{f_d}\opb{f_\pi}\opb{\pi_{M_2}}\omega_{M_2}
&\to&\opb{\pi_{M_1}}\omega_{M_1}.
\eneqn
Hence, we get a map:
\eqn
&&f^\mu\cl \rsect_{\Lambda_2}(\opb{\pi_{M_2}}\omega_{M_2})\to\rsect_{f^\mu\Lambda_2}(\opb{\pi_{M_1}}\omega_{M_1}).
\eneqn
\begin{lemma}
Let $\lambda\in H^0_{\Lambda_1}(T^*M_2;\opb{\pi_{M_2}}\omega_{M_2})$. Then
$ \mueu_{M_{12}}(\cor_{\Gamma_f})\circ\lambda=f^\mu(\lambda)$.
\end{lemma}

\begin{proposition}\label{pro:inversmueu}
Assume that $\tw f$ is non characteristic with respect to $\Lambda_{22}$.
Then the object  $(\cor_1\letens\omega_{M_1/M_2})\ltens\opb{\tw f}K_2$ is a trace kernel on $M_1$ and
\eqn
\mueu_{M_1}\bl\omDA[1]\conv[1]\opb{\tw f}(\omDAI[2]\conv[2]K_2)\br
&=&\mueu_{M_{12}}(\cor_{\Gamma_f})\aconv[2]\mueu_{M_2}(K_2)\\
&=&f^\mu(\mueu_{M_2}(K_2)).
\eneqn
\end{proposition}
\begin{proof}
Applying Theorem~\ref{th:HH1} with $M_3=\rmpt$, we get
that $$
(\cor_1\letens\omega_{M_1/M_2})\ltens\opb{\tw f}K_2
\simeq\HK(\cor_f)\conv[22](\omDAI[2]\conv[2](\omega_2\letens\omega_2^{\otimes-1})\ltens K_2)$$ is a trace kernel.
Since  $\eu_{M_2}\bl(\omega_2\letens\omega_2^{\otimes-1})\ltens K_2)\br
=\mueu_{M_2}(K_2)$ by Proposition~\ref{prop:ambg}, we obtain the result.
\end{proof}

\subsubsection*{Tensor product}
Consider now the case where $M_1=M_2=M$ and the $\Lambda_{ii}$'s satisfy
 the transversality condition
\eq\label{eq:transvlambda}
&&\Lambda_{11}\cap\Lambda_{22}^a\subset T^*_{M\times M}(M\times M).
\eneq
Then
by composing the external product   with the restriction to the diagonal,
we get a  convolution map
\eq
\star\cl\RMH[\Lambda_1]{M}\times\RMH[\Lambda_2]{M}\to\RMH[\Lambda_1+\Lambda_2]{M}.
\eneq
Applying Propositions~\ref{pro:extmueu} and~\ref{pro:inversmueu}, we get:
\begin{proposition}\label{pro:convmueu}
Assume~\eqref{eq:transvlambda}. Then the  object
$K_1\ltens(\cor_M\letens\omAI)\ltens K_2$ is a
trace kernel on $M$ and
\eqn
&&\mueu_M(K_1\ltens(\cor_M\letens\omAI)\ltens  K_2)=\mueu_M(K_1)\star\mueu_M(K_2).
\eneqn
\end{proposition}
Following~\cite[II~Cor.~5.6]{ScSn94}, we shall recall the link between the product $\star$ and the cup product.

\begin{proposition}\label{pro:starcup}
Let $\lambda_i\in H^0_{\Lambda_i}(T^*M_i;\opb{\pi_{M}}\omega_{M})$ \lp$i=1,2$\rp, and
$\Lambda_1\cap\Lambda_2^a\subset T^*_MM$. Then
\eq\label{eq:starcup1}
&&(\lambda_1\star\lambda_2)\vert_M=\int_{\pi_M}(\lambda_1\cup\lambda_2)
\eneq
as elements of $H^0_{\pi(\Lambda_1\cap\Lambda_2)}(M;\omega_M)$.
\end{proposition}
\begin{proof}
Denote by $\delta\cl \Delta\hookrightarrow M_{12}= M\times M$ the diagonal embedding and let us identify
$M$ with $\Delta$. Consider the diagram
\eq\label{diagstar}
&&\xymatrix{
T^*_\Delta M_{12}\ar[d]_-\pi\ar[r]_-f&\Delta\times_{M_{12}}T^*M_{12}\ar[d]^-{\delta_d}\\
\Delta\ar[r]_-s                         &T^*\Delta
}\eneq
where $\pi$ is the projection, $\delta_d$ is the map associated with $\delta$,
$s$ is the zero-section embedding and $f$ is the restriction to $\Delta\times_MT^*M_{12}$
of the embedding $T^*_\Delta M_{12}\hookrightarrow T^*M_{12}$.
Since this diagram is Cartesian, we have
\eqn
&&\opb{s}\eim{\delta_d}\simeq\eim{\pi}\opb{f}.
\eneqn
Now let
$\lambda_1\times\lambda_2\in H^0_{\Lambda_1\times\Lambda_2}(T^*M_{12};\opb{\pi}\omega_{M_{12}})$
and denote by $\lambda_1\times_M\lambda_2$ its image by the map
\eqn
&& H^0_{\Lambda_1\times\Lambda_2}(T^*M_{12};\opb{\pi}\omega_{M_{12}})\to
H^0_{\Lambda_1\times_M\Lambda_2}(\Delta\times_{M_{12}}T^*M_{12};\opb{\pi}\omega_{M_{12}}).
\eneqn
(Here, on the right hand side, we still denote by $\pi$ the restriction of the projection $\pi_{M_{12}}$ to
$\Delta\times_{M_{12}}T^*M_{12}$.)
Then
\eqn
&&\int_\pi(\lambda_1\cup\lambda_2)=\eim{\pi}\opb{f}(\lambda_1\times_M\lambda_2),\\
&&(\lambda_1\star\lambda_2)\vert_M=\opb{s}\eim{\delta_d}(\lambda_1\times_M\lambda_2).
\eneqn
\end{proof}

\begin{corollary}\label{cor:abstrindexth}
Let $K_1$ and $K_2$ be two trace kernels on $M$ with $\SSi(K_i)\subset\Lambda_{ii}$.
Assume~\eqref{eq:transvlambda} and assume moreover that
$\Supp(K_1)\cap\Supp(K_2)$ is compact. Then the object
$\rsect\bl M\times M;K_1\ltens(\cor_M\letens\omAI)\ltens  K_2\br$  is a trace
kernel on $\rmpt$ and
\eqn
&&\eu_\rmpt\bl\rsect(M;K_1\ltens(\cor_M\letens\omAI)\ltens K_2\br
=\int_{T^*M}\mueu(K_1)\cup\mueu(K_2).
\eneqn
\end{corollary}
\begin{remark}
Let $M$  be a real analytic manifold and let $F\in\Derb_\Rc(\cor_{M})$. Recall that one associates to $F$ the
trace kernel $\HK(F)=F\letens\RD_MF$ and that $\mueu_M(F)=\mueu_M(\HK(F))$. Assume now that $f\cl M_1\to M_2$
is a morphism of real analytic manifolds.

\medskip\noindent
Let $F_1\in\Derb_\Rc(\cor_{M_1})$ and assume that $f$ is proper on $\Supp(F_1)$.
Applying Proposition~\ref{pro:directmueu} and noticing that
\eq
&&\reim{\tw f}\HK(F_1)\simeq\HK(\reim{f}F_1),
\eneq
 we find
that $\mueu(\reim{f}F_1)=f_\mu(\mueu(F_1))$.
It is nothing but \cite[Prop.~9.4.2]{KS90}.

\medskip\noindent
Let $F_2\in\Derb_\Rc(\cor_{M_2})$ and assume that $f$ is non characteristic with respect to $F_2$.
Applying Proposition~\ref{pro:inversmueu}  and noticing that
\eqn
&&\HK(f^{-1}F_2)
\simeq(\cor_1\letens\omega_{M_1/M_2})\ltens
\opb{\tw f}\HK(F_2),
\eneqn
we find that
$\mueu(\opb{f}F_2)=f^\mu(\mueu(F_2))$.
Hence, we recover~\cite[Prop.~9.4.3]{KS90}.
\end{remark}

\section{Applications: $\shd$-modules and elliptic pairs}
As an application of Theorem~\ref{th:HH1}, we shall recover the theorem of~\cite{ScSn94} on the
index of elliptic pairs.
In this section, $X$ is a complex manifold, $\cor=\C$, $\shm$ is an object of $\Derb_\coh(\shd_X)$ and $F$
is an object of $\Derb_\Rc(\C_X)$.

Recall that we have denoted by $\HK(F)$ and $\HK(\shm)$ (see Notation~\ref{not:HKM})
the trace kernels associated with  $F$ and with $\shm$,
 respectively:
\eqn
\HK(F)&\seteq&F\letens\RD_XF,\\
\HK(\shm)&\seteq& \Omega_{X\times X}\lltens[\shd_{X\times X}](\shm\detens\RDD\shm).
\eneqn
The pair $(\shm,F)$ is called an {\em elliptic pair} in loc.\ cit.\  if $\chv(\shm)\cap\SSi(F)\subset T^*_XX$.
{}From now on, we assume that  $(\shm,F)$ is an elliptic pair.

It follows from Proposition~\ref{pro:convmueu} that the tensor product of $\HK(F)$ and $\HK(\shm)$ shifted
by $-2d_X$ is again a trace kernel. We denote it by  $\HK(\shm,F)$. Hence
\eq
&& \HK(\shm,F)\simeq \Omega_{X\times X}\lltens[\shd_{X\times X}](\shm\detens\RDD\shm)\tens(F\letens\RD'_XF).
\eneq
Moreover the same statement  gives:
\eq\label{eq:ellpairs2}
&&\mueu_X\bl\HK(\shm,F)\br=\mueu_X(\shm)\star\mueu_X(F).
\eneq
We set
\eq
&&\GSol(\shm,F)\eqdot \RHom[\shd_X](\shm\tens F,\sho_X),\label{eq:SolMF}\\
&&\DR(\shm,F)\eqdot \rsect(X;\Omega_X\lltens[\shd_{X}]\shm\tens F)\,[d_X].\label{eq:DRMF}
\eneq

As explained in~\cite{ScSn94}, Theorem~\cite[Th~11.3.3]{KS90} and isomorphism~\eqref{eq:elliptshv}
provides
a generalization of the classical Petrovsky
regularity theorem, namely, the natural isomorphisms
\eq\label{eq:petrovski}
&&\rhom[\shd_X](\shm,\RD'_XF\tens\sho_X)
\isoto\rhom[\shd_X](\shm\tens F,\sho_X).
\eneq
Now assume that  $\Supp(\shm)\cap\Supp(F)$ is compact and let us take the global sections
of the isomorphism~\eqref{eq:petrovski}.
We find the isomorphism
\eq\label{eq:petrovski5}
&&\RHom[\shd_X](\shm,\RD'_XF\tens\sho_X)\isoto\RHom[\shd_X](\shm\tens F,\sho_X).
\eneq
It is proved in~\cite{ScSn94}
\footnote{In fact, the finiteness of the cohomology of
 this complex is only proved in  loc.\ cit.\ under the hypothesis that $\shm$ admits a good filtration,
but this hypothesis may be removed thanks to the results of~\cite[Appendix]{KS96}.}
that one can represent the left hand side
of~\eqref{eq:petrovski5} by a complex
of topological vector spaces of type DFN and the right hand side of~\eqref{eq:petrovski5} by a complex
of topological vector spaces of type FN. It follows that
the complexes $\GSol(\shm,F)$ and $\DR(\shm,F)$ have
finite dimensional cohomology and are dual to each other. More precisely,
denoting by $(\scbul)^*$ the duality functor in $\Derb_f(\C)$, we have
\eqn
&&\bl\GSol(\shm,F)\br^*\simeq\DR(\shm,F).
\eneqn
It follows from the finiteness of the cohomology of the complexes  $\GSol(\shm,F)$ and $\DR(\shm,F)$ that
\eqn
\rsect(X\times X;\HK(\shm,F))&\simeq&\GSol(\shm,F)\tens \DR(\shm,F).
\eneqn

One checks that this isomorphism commutes with the composition of the morphisms
$\C\to\rsect(X\times X;\HK(\shm,F)) \to \C$ and $\C\to \GSol(\shm,F)\tens \DR(\shm,F)\to\C$, which implies
\eq
&&\eu_\rmpt\bl\rsect(X\times X;\HK(\shm,F)\br=\chi\bl\GSol(\shm,F)\br.
\eneq
 Therefore, one recovers the index formula of loc.\ cit.\
\eq\ba{rcl}\label{eq:RR}
\chi\bl\RHom[\shd_X](\shm\tens F,\sho_X)\br&=&\int_X(\mueu_X(\shm)\star\mueu_X(F))\vert_X\\[1ex]
&=&\int_{T^*X}\mueu_X(\shm)\cup\mueu_X(F).
\ea\eneq
\begin{remark}
In general the direct image of an elliptic pair is no more an elliptic pair. However, it remains a trace kernel.
\end{remark}
\begin{remark}
As already mentioned in~\cite{ScSn94}, formula~\eqref{eq:RR} has many applications,
as far as one is able to calculate
$\mueu_X(\shm)$ (see the final remarks below). For example, if $M$ is a compact real analytic
manifold and $X$ is a complexification of $M$, one recovers the Atiyah-Singer theorem by choosing
$F=\RD'\C_M$. If $X$ is a complex compact manifold, one recovers the Riemann-Roch theorem: one takes
$F=\C_X$ and if $\shf$ is a coherent $\sho_X$-module, one sets
$\shm=\shd_X\tens[\sho_X]\shf$.
\end{remark}

\section{The Lefschetz fixed point formula}
In this section, we shall briefly show how to adapt the formalism of trace kernels to the
Lefschetz trace formula as treated in~\cite[\S~9.6]{KS90}. Here we assume that $\cor$ is a field.

Assume to be given two maps $f,g\cl N\to M$ of real analytic manifolds, an object $F\in\Derb_\Rc(\cor_M)$ and a morphism
\eq\label{eq:lef1}
&&\phi\cl\opb{f}F\to\epb{g}F.
\eneq
Set
\eqn
&&h=(g,f)\cl N\times N \to M\times M,\\
&&S=\Supp(F), 
\quad L=\opb{h}(\Delta_M)=\set{(x,y)\in N\times N}{g(x)=f(y)},\\
&& i\cl L\hookrightarrow N\times N,\\
&&T=\opb{f}(S)\cap\opb{g}(S).
\eneqn
One makes the assumption
\eq\label{hyp:lef2}
&&\mbox{The set $T$ is compact.}
\eneq
Then we have the maps
\eqn
&&\rsect(M;F)\to\rsect_{\opb{f}S}(N;\opb{f}F)\to[\phi] \rsect_{T}(N;\epb{g}F)\to\rsect(M;F).
\eneqn
The composition gives a map
\eq\label{eq:lef3}
&&\int\phi\cl\rsect(M;F)\to\rsect(M;F),
\eneq
and this map factorizes through $\rsect_{T}(N;\epb{g}F)$ which has
 finite-dimensional cohomologies. Hence, we can define the trace
$\tr(\int\phi)$.

 We have the chain of morphisms 
\eqn
\cor_N&\to&\rhom(\epb{g}F,\epb{g}F)\\
&\to[\phi]&\rhom(\opb{f}F,\epb{g}F)
\simeq\epb{\delta_N}(\epb{g}F\letens\RD_N\opb{f}F)\\
&\simeq&\epb{\delta_N}(\epb{g}F\letens\epb{f}\RD_MF)\simeq \epb{\delta_N}\epb{h}(F\letens \RD_MF).
\eneqn
We have thus constructed the morphism 
\eqn
&&\dA[N]\to \epb{h}(F\letens \RD_MF).
\eneqn

By using the morphism $F\letens \RD_MF\to \omDA[M]$ and the isomorphism $\epb{h}\omDA[M]\simeq \oim{i}\omA[L]$,
we get  the morphisms
\eq\label{eq:Kphi}
&&\dA[N]\to \epb{h}(F\letens \RD_MF)\to i_*\omA[L]
\eneq
in $\Derb(\cor_{N\times N})$.
The support of the composition is contained in $\delta_N(T)\cap L$.
\begin{theorem}[{\cite[Proposition~9.6.2]{KS90}}]
The trace
$\tr(\int\phi)$ coincides with the image of $1\in\cor$
by the composition of the morphisms
\eqn
&&\cor\to\rsect(N,\cor_N)
\to\rsect_c(L,\omA[L])\to \cor.
\eneqn
Here the middle arrow is derived from \eqref{eq:Kphi}.
\end{theorem}

 Although~\eqref{eq:Kphi} is not a trace kernel in the sense of Definition~\ref{def:traceker}, 
it should be possible to adapt the previous constructions to the case of $\shd$-modules and
to elliptic pairs, then to
recover a theorem of~\cite{Gu96} but we do not develop this point here (see~\cite{RTT12} for related results).

\subsubsection*{Final remarks}
The microlocal Euler class of constructible sheaves is easy to compute since it is enough to calculate
some multiplicities at generic points. We refer to~\cite{KS90} for examples.

On the other hand, there is no direct  method to calculate the  microlocal Euler class of
 a coherent $\shd$-module $\shm$ (except in the holonomic case).
In \cite{ScSn94}, the authors made a precise conjecture relying $\mueu_X(\shm)$ and the Chern
character of the associated graded module (an $\sho_{T^*X}$-module), and this conjecture has been proved
 by Bressler-Nest-Tsygan~\cite{BNT02}.

Similarly, the Hochschild class of coherent $\sho_X$-modules is usually calculated through the
 so-called Hochschild-Kostant-Rosenberg isomorphism, 
but this isomorphism does not commute with
 proper direct images, and a precise conjecture (involving the Todd class) has been made by
Kashiwara in~\cite{Ka91} and this conjecture has recently been proved in the algebraic case by
 Ramadoss~\cite{Ra08} and in the general case by Grivaux~\cite{Gr09}.

\providecommand{\bysame}{\leavevmode\hbox to3em{\hrulefill}\thinspace}

\vspace*{1cm}
\noindent
\parbox[t]{25em}
{\scriptsize{
\noindent
Masaki Kashiwara\\
Research Institute for Mathematical Sciences, Kyoto University \\
and Department of Mathematical Sciences, Seoul National
University\\
e-mail: masaki@kurims.kyoto-u.ac.jp

\medskip\noindent
Pierre Schapira\\
Institut de Math{\'e}matiques,
Universit{\'e} Pierre et Marie Curie\\
and Mathematics Research Unit,
University of Luxemburg\\
e-mail: schapira@math.jussieu.fr\\
}}

\end{document}